\documentclass[a4paper,11pt]{article}
\usepackage[normalem]{ulem}
\usepackage{amsmath,amsthm,amssymb,enumerate}
\usepackage[T1]{fontenc}
\usepackage{gensymb}
\usepackage[authoryear,round]{natbib}
\usepackage[sc]{mathpazo}
\usepackage{multirow}
\usepackage{newtxtext,newtxmath}
\usepackage{subfig}
\usepackage{graphicx}
\usepackage{epstopdf}
\usepackage{cancel}
\usepackage[margin=3cm]{geometry}
\usepackage{tikz}
\usepackage{caption}
\usetikzlibrary{shapes,calc}
\usepackage{verbatim}
\usepackage{mathrsfs}
\usepackage{accents}
\usepackage[utf8]{inputenc}
\usepackage{booktabs}
\usepackage{xcolor}

\usetikzlibrary{decorations.pathmorphing,decorations.pathreplacing,
positioning,shapes,arrows,patterns,fadings,plotmarks,calc,intersections}
\tikzstyle{every picture}+=[font=\footnotesize]
\usepackage{paralist}
\usepackage{bbm}
\usepackage{latexsym}
\usepackage{enumerate}
\usepackage{enumitem}

\setlist{noitemsep, topsep=0.8ex, partopsep=0pt, leftmargin=3em}
\setlist[1]{labelindent=\parindent}
\newlist{axioms}{enumerate}{1}
\setlist[axioms]{font=\bfseries}
\newlist{alphenum}{enumerate}{1}
\setlist[alphenum]{label=\textbf{(\alph*)}, leftmargin=4em}
\newlist{alphienum}{enumerate}{1}
\setlist[alphienum]{label=\textit{(\alph*)}}
\newlist{romanenum}{enumerate}{1}
\setlist[romanenum]{label=\textit{(\roman*)}}
\newlist{romaninenum}{enumerate*}{1}
\setlist[romaninenum]{label=\textit{(\roman*)}}

\usepackage[vlined]{algorithm2e}
\SetKwIF{If}{ElseIf}{Else}{if}{}{else if}{else}{endif}
\SetKwFor{For}{for}{}{endfor}

\usepackage[noabbrev, capitalise]{cleveref}
\usepackage{tabu}
\crefname{equation}{\unskip}{\unskip}
\creflabelformat{equation}{#2(#1)#3}

\SetFuncSty{textsc}
\SetKwFunction{frun}{Run}\SetKwHangingKw{arun}{\frun}
\SetKwFunction{fset}{Set}\SetKwHangingKw{aset}{\fset}
\SetKwFunction{ffind}{Find}\SetKwHangingKw{afind}{\ffind}
\SetKwFunction{fselect}{Select}\SetKwHangingKw{aselect}{\fselect}
\SetKwFunction{fcompute}{Compute}\SetKwHangingKw{acompute}{\fcompute}
\SetKwFunction{fsolve}{Solve}\SetKwHangingKw{asolve}{\fsolve}
\SetKwFunction{fupdate}{Update}\SetKwHangingKw{aupdate}{\fupdate}
\SetKwFunction{festimate}{Estimate}\SetKwHangingKw{aestimate}{\festimate}
\SetKwFunction{fmark}{Mark}\SetKwHangingKw{amark}{\fmark}
\SetKwFunction{fbreak}{Break}\SetKwHangingKw{abreak}{\fbreak}
\SetKwFunction{frefine}{Refine}\SetKwHangingKw{arefine}{\frefine}
\SetKwFunction{compute}{Compute}
\SetKwFunction{set}{Set}

{\algorithm}%
{\endalgorithm}

\newtheorem{thm}{Theorem}[section]

\newtheorem{lem}[thm]{Lemma}

\theoremstyle{definition}

\theoremstyle{remark}
\newtheorem{rem}{Remark}[section]
\newtheorem{example}{\bf Example}[section]

\numberwithin{equation}{section}

\newcommand{\cA}{\mathcal A}

\newcommand{\cC}{\mathcal C}

\newcommand{\cV}{\mathcal V}

\newcommand{\cM}{\mathcal M}

\newcommand{\vket}{von K\'{a}rm\'{a}n equations }
\newcommand{\vk}{von K\'{a}rm\'{a}n}

\newcommand{\sit}{\sum_{T \in\mathcal{T}_h}\int_T}

\newcommand{\fl}{\;\text{ for all }}

\newcommand{\half}{\frac{1}{2}}
\newcommand{\trinl}{\ensuremath{|\!|\!|}}
\newcommand{\trinr}{\ensuremath{|\!|\!|}}

\newcommand{\dx}{{\rm\,dx}}

\newcommand{\ds}{{\rm\,ds}}

\DeclareMathOperator{\E}{\mathcal{E}}

\newcommand{\pw}{\text{pw}}

\newcommand{\T}{\mathcal{T}}

\newcommand{\bdalpha}{{\boldsymbol{\alpha}}}
\newcommand{\bdbeta}{{\boldsymbol{\beta}}}

\DeclareMathOperator*{\argmin}{arg\,min}

\def \R{{{\Bbb R}}}

\allowdisplaybreaks

\def\R{\mathbb{R}}

\def\cA{\mathcal{A}}

\def\p{\partial}
\def\O{\Omega}

\def\cV{\mathcal{V}}

\def\cM{\mathcal{M}}

\def\smean#1{\{\hskip -3pt\{#1\}\hskip -3pt\}}
\def\sjump#1{[\hskip -1.5pt[#1]\hskip -1.5pt]}

\title{A Quadratic $C^0$ Interior Penalty Method for the von K\'{a}rm\'{a}n Obstacle Problem}
\author{Sharat Gaddam\thanks{High Energy Materials Research Laboratory, DRDO,
Pune 411021, Maharashtra, India}\thanks{E-mail: gaddamsharat@gmail.com}}
\date{}

\begin{document}
\maketitle


\providecommand{\checkmark}{\ensuremath{\surd}}
\makeatletter
\@ifundefined{c@example}{\newtheorem{example}{Example}[section]}{}
\makeatother


\begin{abstract}
This article proposes and analyses a quadratic $C^0$ interior penalty method
for the displacement obstacle problem of the von K\'arm\'an plate. The discrete
space consists of Lagrange $P_2$ finite elements and the obstacle constraint is
imposed at the vertices. The trilinear form of the von K\'arm\'an bracket is
modified by terms on the edges so that it is bounded in the discrete energy
norm. The well-posedness of the discrete problem, namely the existence of a
discrete solution and its uniqueness under a smallness condition on the data,
is established. The Sobolev and Friedrichs constants of the discrete energy
norm are quantified with an explicit dependence on the mesh size. The main result
is an error estimate of order $\mathcal{O}(h^{\alpha})$ in the discrete energy
norm, where $1/2<\alpha\le1$ is the index of elliptic regularity of the
biharmonic operator on the polygonal domain. Numerical experiments on a square
and on an L-shaped domain confirm the predicted rates. The coincidence set has positive measure in one example and empty interior in another. The experiments also show how the penalty parameter affects the rates and identify a threshold in the size of the obstacle beyond which the iterative solver fails on fine meshes.
\end{abstract}



\section{Introduction}\label{sec:intro}

\subsection{Problem formulation}

The \vket~\citep{ciarlet1997mathematical} describe the bending of a thin elastic
plate by a system of two fourth-order semilinear equations. The unknowns are
the transverse displacement and the Airy stress function. Existence, regularity
and bifurcation for this system are studied in
\citep{berger1968karman,knightly1967existence,ciarlet1997mathematical}. This
article treats the case in which the plate lies above a given obstacle. The
displacement then solves a variational inequality and the contact set on which the
plate meets the obstacle is unknown.

Let $\O\subset\R^2$ be a bounded polygonal Lipschitz domain and let
$\chi\in H^2(\O)$ satisfy $\max_{x\in\partial\O}\chi(x)<0$. The set
\[
K:=\{\varphi\in H^2_0(\O)\,:\,\varphi\ge\chi\text{ a.e.\ in }\O\}
\]
is non-empty, closed and convex in $H^2_0(\O)$. Let $D^2\varphi$ denote the
Hessian and let
$[\varphi_1,\varphi_2]:=\varphi_{1xx}\varphi_{2yy}+
\varphi_{1yy}\varphi_{2xx}-2\varphi_{1xy}\varphi_{2xy}$
denote the von K\'arm\'an bracket. Define
\begin{align}\label{defnaandb}
a(\varphi_1,\varphi_2)&:=(D^2\varphi_1,D^2\varphi_2)_{L^2(\O)},
&
b(\varphi_1,\varphi_2,\varphi_3)&:=
-\tfrac12\bigl([\varphi_1,\varphi_2],\varphi_3\bigr)_{L^2(\O)}
\end{align}
for $\varphi_1,\varphi_2,\varphi_3\in H^2_0(\O)$. The form
$b(\bullet,\bullet,\bullet)$ is symmetric in all three arguments
\citep[Cor.~2.3]{brenner2017c}. Given $f\in L^2(\O)$, the von K\'arm\'an
obstacle problem seeks the displacement $u\in K$ and the Airy stress function
$v\in H^2_0(\O)$ with
\begin{subequations}\label{wform}
\begin{align}
a(u,u-\varphi_1)+2b(u,v,u-\varphi_1)
&\le(f,u-\varphi_1)_{L^2(\O)}
&&\fl \varphi_1\in K,\label{wforma}\\
a(v,\varphi_2)-b(u,u,\varphi_2)
&=0
&&\fl \varphi_2\in H^2_0(\O).\label{wformb}
\end{align}
\end{subequations}
The obstacle constrains the displacement only. The existence of a solution of
\eqref{wform}, its uniqueness under a smallness condition on the data and its
regularity on a polygonal domain are established in
\citep{carstensen2021morley}; \Cref{sec:wp} states these results.

\subsection{Literature}

Finite element methods for the von K\'arm\'an equations without a constraint
are analysed for conforming and mixed discretisations in
\citep{brezzi1978finite,miyoshi1976mixed,reinhart1982numerical,mallik2016conforming},
for hybrid discretisations in \citep{quarteroni1979hybrid}, for the Morley
element in \citep{mallik2016nonconforming,carstensen2017nonconforming}, and for
discontinuous Galerkin and $C^0$ interior penalty methods in
\citep{brenner2017c,carstensen2018priori}. Adaptive versions are studied in
\citep{carstensen2019adaptive}.

Fourth-order variational inequalities arise in contact problems for plates. The
theory of variational inequalities is presented in
\citep{glowinski2008lectures,kinderlehrer1980introduction}. The regularity of
the solution and of the free boundary is studied in
\citep{frehse1971differenzierbarkeitsproblem,frehse1973regularity,caffarelli1979obstacle}.
The displacement obstacle problem of the clamped Kirchhoff plate is analysed in
\citep{brenner2012finite,brenner2013morley,brenner2012quadratic} for $C^1$
conforming elements, for $C^0$ interior penalty methods and for classical
nonconforming elements. Numerical algorithms for such problems are developed in
\citep{hintermuller2002primaldual,hoppe1994adaptive,suttmeier2008numerical}.

Unilateral problems for related plate models are studied in
\citep{muradova2007unilateral,yau1992obstacle,miersemann1992stability}, and a
conforming penalty method in
\citep{ohtake1980analysisI,ohtake1980analysisII}. The article \citep{carstensen2021morley} establishes the
existence of a solution of \eqref{wform}, its uniqueness under an {\em a priori} and
an a posteriori smallness condition on the data, and the regularity of the
solution on a polygonal domain, and it analyses the Morley finite element
approximation of \eqref{wform}.

\subsection{Main results}

This article proposes the quadratic $C^0$ interior penalty method for
\eqref{wform} and proves its convergence. The method is defined on $V_h=\mathcal{P}_2(\T_h)\cap H^1_0(\O)$. Two
properties motivate this space. Its members are continuous, so the constraint
$\chi(p)\le\varphi_h(p)$ at the vertices $p$ defines a discrete admissible set
of continuous functions. The inclusion $V_h\subset H^1_0(\O)$ leaves only the
normal derivative discontinuous across an interior edge, so the consistency and
penalty terms of the method involve the single jump
$\sjump{\p\varphi_h/\p n}$. The $C^0$ interior penalty method is analysed for
the von K\'arm\'an equations in \citep{brenner2017c} and for the biharmonic
obstacle problem in \citep{brenner2012quadratic}, and the present article
combines the two analyses.

\Cref{sec:disc-setup} defines the bilinear form $a_{\rm IP}$ and the trilinear
form $b_{\rm IP}$. The form $b_{\rm IP}$ contains terms on the edges in
addition to the piecewise bracket. These terms make $b_{\rm IP}$ bounded in the
discrete energy norm and give the identities of \Cref{lem:trilinearprop}.
\Cref{thm:disc-sob} gives the Sobolev and Friedrichs constants of the discrete
energy norm with an explicit power of the mesh size. \Cref{thm:disc-existence} proves
the well-posedness of the discrete problem, namely the existence of a
discrete solution and its uniqueness under a smallness condition on the data. \Cref{thm:err} proves
\[
\|u-u_h\|_h+\|v-v_h\|_h\lesssim h^{\alpha}
\]
under the same condition. \Cref{sec:numerics} reports numerical experiments on
a square and on an L-shaped domain.

The results of \citep{carstensen2021morley} on the continuous problem
\eqref{wform} are used throughout this article and are stated without proof in
\Cref{sec:wp}. That article discretises \eqref{wform} with Morley elements.
The present discrete space is conforming in
$H^1(\O)$, and the fourth-order operator is stabilised by consistency and
penalty terms on the edges with a parameter $\sigma$. Both methods converge
with the order $\mathcal{O}(h^\alpha)$.

\subsection{Notation and outline}

Standard notation for Lebesgue and Sobolev spaces applies. For $s>0$ and
$1\le p\le\infty$, the seminorm and the norm on $H^s(\O)$ are $|\bullet|_s$ and
$\|\bullet\|_s$, those on $W^{s,p}(\O)$ are $|\bullet|_{s,p}$ and
$\|\bullet\|_{s,p}$, and $\|\bullet\|_{-s}$ is the norm on the dual space. The
$L^2(\O)$ scalar product and norm are $(\bullet,\bullet)_{L^2(\O)}$ and
$\|\bullet\|_{L^2(\O)}$, and $H^{-2}(\O)$ is the dual of $H^2_0(\O)$. The
energy norm is $\trinl\bullet\trinr:=\|D^2\bullet\|_{L^2(\O)}$ and its
piecewise version is
$\trinl\bullet\trinr_{\pw}:=\|D^2_{\pw}\bullet\|_{L^2(\O)}$, with the piecewise
bracket $[\bullet,\bullet]_{\pw}$. The index of elliptic regularity of the
biharmonic operator on $\O$ is $\alpha\in(1/2,1]$ and depends on the largest
interior angle of $\O$ \citep{blum1980boundary}. It equals $1$ for a convex
polygon and approximately $0.544$ for the L-shaped domain. The notation
$A\lesssim B$ abbreviates $A\le CB$ with a positive constant $C$ that depends
on $\trinl u\trinr$, $\trinl v\trinr$, $\|u\|_{2+\alpha}$,
$\|v\|_{2+\alpha}$ and $\|f\|_{L^2(\O)}$ and is independent of the mesh size,
and $A\approx B$ abbreviates $A\lesssim B\lesssim A$. The constants
$C_{\rm S}$ and $C_{\rm F}$ in the embeddings
$H^2_0(\O)\hookrightarrow C(\overline\O)$ and
$H^2_0(\O)\hookrightarrow L^2(\O)$ satisfy
\begin{align}\label{ctsembedding}
\|v\|_{L^\infty(\O)}\le C_{\rm S}\|v\|_{H^2_0(\O)},\quad
\|v\|_{L^2(\O)}\le C_{\rm F}\|v\|_{H^2_0(\O)}\fl v\in H^2_0(\O).
\end{align}

\Cref{sec:wp} states the properties of \eqref{wform} that the analysis uses.
\Cref{sec:disc-setup} defines the discrete forms and the discrete energy norm and
states the interpolation, enrichment and embedding results.
\Cref{sec:disc-prob} defines the discrete problem and proves
\Cref{thm:disc-existence}. \Cref{sec:err-est} proves \Cref{thm:err}.
\Cref{sec:numerics} presents the numerical experiments and
\Cref{sec:conclusion} concludes.

\section{The continuous problem}\label{sec:wp}

The existence of a solution of \eqref{wform}, its uniqueness under a smallness
condition on the data and its regularity are established in
\citep{carstensen2021morley}. This section states the results that the discrete
analysis uses and gives no proofs. \Cref{lem:Bound-b} is
\citep[Lem.~2.1]{brezzi1978finite}, \Cref{thm:contsdep} combines
\citep[Thm.~2.3]{carstensen2021morley} with
\citep[Rem.~2.1]{carstensen2021morley}, and \Cref{thm:reg} is
\citep[Thm.~3.5]{carstensen2021morley}.

\begin{lem}[Bound for $b(\bullet,\bullet,\bullet)$
\citep{brezzi1978finite}]\label{lem:Bound-b}
\noindent For any $\varphi_1,\varphi_2,\varphi_3\in H^2_0(\O)$, the trilinear
form $b(\bullet,\bullet,\bullet)$ introduced in \eqref{defnaandb} admits the
bound
$$b(\varphi_1,\varphi_2,\varphi_3)\le
\trinl\varphi_1\trinr\trinl\varphi_2\trinr\|\varphi_3\|_{L^\infty(\O)}
\le C_{\rm S}\trinl\varphi_1\trinr\trinl\varphi_2\trinr\trinl\varphi_3\trinr.$$
\end{lem}

\begin{thm}[Existence, {\em a priori} bound, {\em a priori} and
{\em a posteriori} criteria for uniqueness
\citep{carstensen2021morley}]\label{thm:contsdep}
\noindent Let $(f,\chi)\in L^2(\O)\times H^2(\O)$ satisfy
$\max\chi(\partial\O)<0$. Then the problem \eqref{wform} admits at least one
solution $(u,v)\in K\times H^2_0(\O)$. Moreover, with a positive constant
$C(\chi)$ determined by the obstacle alone, every solution $(u,v)$ of
\eqref{wform} obeys $(a)$-$(b)$.
\begin{itemize}
\item[(a)] $\dfrac{1}{2}\trinl u\trinr^2+\trinl v\trinr^2\le N^2(f,u):=
\trinl u\trinr^2+\trinl v\trinr^2-2(f,u)_{L^2(\O)}+2C^2_{\rm F}\|f\|^2\\
\le M^2(f,\chi):=C(\chi)+3C^2_{\rm F}\|f\|^2_{L^2(\O)},$
\item[(b)] If $\dfrac{C_{\rm S}^2}{4}\trinl u\trinr^2+
C_{\rm S}\trinl v\trinr<\half$, then $(u,v)$ is the only solution to
\eqref{wform}.
\end{itemize}
The {\em a priori} smallness assumption $C_{\rm S}M(f,\chi)<1/2$ on the data,
as well as the {\em a posteriori} criterion $C_{\rm S}N(f,\varphi)<1/2$ for a
single $\varphi\in K$, each yield $C_{\rm S}N(f,u)<1/2$ and therefore the
uniqueness of the solution to \eqref{wform}.
\end{thm}

\begin{thm}[\textbf{Regularity for the \vk\, obstacle problem}
\citep{carstensen2021morley}]\label{thm:reg}
\noindent Let $\O$ be a bounded polygonal domain in $\R^2$ and let the data
satisfy $f\in L^2(\O)$ and
$\chi\in H^2(\O)\cap H^3_{\rm loc}(\O)\cap C^2(\O)$ with
$\max\chi(\partial\O)<0$. Then every solution $(u,v)\in K\times H^2_0(\O)$ of
\eqref{wform} enjoys the regularity
$u,v\in H^{2+\alpha}(\O)\cap H^3_{\rm loc}(\O)\cap C^2(\O)$ for the index
$\alpha\in(1/2,1]$ of elliptic regularity. If, in addition, the bounded
Lipschitz domain $\O$ possesses a $C^{2+\gamma}$ boundary for some
$0<\gamma<1$, then $u,v\in C^2(\overline{\O})$.
\end{thm}

\section{The discrete framework}\label{sec:disc-setup}

\subsection{Triangulations and broken spaces}

Let $\T_h$ be a shape-regular admissible triangulation of $\O$ into closed
triangles. Let $h_T:=\mathrm{diam}(T)$ for $T\in\T_h$ and
$h_{\max}:=\max_{T\in\T_h}h_T$. The sets of edges, vertices and edge midpoints
are $\E_h$, $\cV_h$ and $\cM_h$, and $\E_h$ is the union of the interior edges
$\E_h^i$ and the boundary edges $\E_h^b$. For $r\in\mathbb{N}_0$ let
$\mathcal{P}_r(T)$ be the space of polynomials of total degree at most $r$ on
$T$ and set
\[
\mathcal{P}_r(\T_h):=\{\phi\in L^2(\O)\,:\,\phi|_T\in\mathcal{P}_r(T)
\ \forall T\in\T_h\},\quad
H^s(\O,\T_h):=\{\phi\in L^2(\O)\,:\,\phi|_T\in H^s(T)\ \forall T\in\T_h\}.
\]
Let $T_-$ and $T_+$ be the two triangles that share an interior edge
$e\in\E_h^i$ and let the unit normal $n_e$ point out of $T_+$. For
$\phi\in H^1(\O,\T_h)$ set $\phi_\pm:=\phi|_{T_\pm}$ and
\[
\smean{\phi}_e:=\tfrac12(\phi_-+\phi_+),\qquad
\sjump{\phi}_e:=\phi_+-\phi_- .
\]
For $e\in\E_h^b$ set $\smean{\phi}_e:=\phi|_e$ and $\sjump{\phi}_e:=\phi|_e$.
These definitions apply componentwise to vector-valued and matrix-valued
functions.

\subsection{The discrete forms}

The discrete space is
\[
V_h:=\mathcal{P}_2(\T_h)\cap H^1_0(\O).
\]
Its members vanish on $\partial\O$, and the condition $\p u/\p n=0$ is imposed
weakly. Let $\sigma>0$ be the penalty parameter. For
$\psi_h,\varphi_h,w_h\in V_h$ define
\begin{align*}
a_{\rm pw}(\psi_h,\varphi_h)&:=\sit D^2\psi_h:D^2\varphi_h\,\dx,\\
a_{\rm aj}(\psi_h,\varphi_h)&:=
\sum_{e\in\E_h}\!\int_e\!\smean{D^2\psi_h\,n_e}\!\cdot\!
\sjump{\nabla\varphi_h}\,\ds
+\sum_{e\in\E_h}\!\int_e\!\smean{D^2\varphi_h\,n_e}\!\cdot\!
\sjump{\nabla\psi_h}\,\ds,\\
a_{\rm jj}(\psi_h,\varphi_h)&:=
\sum_{e\in\E_h}\frac{\sigma}{|e|}\int_e
\sjump{\nabla\psi_h\cdot n_e}\,
\sjump{\nabla\varphi_h\cdot n_e}\,\ds,\\
b_{\rm pw}(\psi_h,\varphi_h,w_h)&:=
-\tfrac12\sit[\psi_h,\varphi_h]w_h\,\dx,\\
b_{\rm j}(\psi_h,\varphi_h,w_h)&:=
-\tfrac12\sum_{e\in\E_h}\int_e\smean{cof(D^2 \psi_h)}\,
\sjump{\nabla \varphi_h}\cdot n_e\,w_h\,\ds
\end{align*}
and set
\begin{align}\label{defn-ah}
a_{\rm IP}(\psi_h,\varphi_h)
&:=a_{\rm pw}(\psi_h,\varphi_h)
-a_{\rm aj}(\psi_h,\varphi_h)
+a_{\rm jj}(\psi_h,\varphi_h),\\[2pt]
\label{defn-bh}
b_{\rm IP}(\psi_h,\varphi_h,w_h)
&:=b_{\rm pw}(\psi_h,\varphi_h,w_h)
-b_{\rm j}(\psi_h,\varphi_h,w_h)
-b_{\rm j}(\varphi_h,\psi_h,w_h).
\end{align}
The sums in $a_{\rm aj}$ and $a_{\rm jj}$ include the boundary edges. The form
$b_{\rm IP}(\bullet,\bullet,\bullet)$ is symmetric in its first two arguments,
and \Cref{lem:bIPbound} bounds it in the discrete energy norm.

Every $\varphi_h\in V_h$ has a continuous tangential derivative across each
edge, so $\sjump{\nabla\varphi_h}=\sjump{\p\varphi_h/\p n}\,n_e$. With the
identity $n_e^{\!\top} cof(H)\,n_e=t_e^{\!\top}H\,t_e$ for a symmetric matrix
$H$, the forms $a_{\rm aj}$ and $b_{\rm j}$ reduce to scalar integrals over the
edges.

The discrete energy norm is
\begin{equation}\label{discretenorm}
\|\varphi_h\|_h^2
:=\sum_{T\in\T_h}|\varphi_h|_{H^2(T)}^2
+\sum_{e\in\E_h}\sum_{j,k=1}^2|e|\,
\bigl\|\smean{\p^2\varphi_h/\p x_j\p x_k}\bigr\|_{L^2(e)}^2
+\sum_{e\in\E_h}\frac{\sigma}{|e|}\,
\bigl\|\sjump{\nabla\varphi_h\cdot n_e}\bigr\|_{L^2(e)}^2 .
\end{equation}
{\bf Note:} The seminorm $\|\bullet\|_h$ is well defined on
$\mathcal{P}_2(\T_h)+H^{2+\alpha}_0(\O)$, and it is a norm on $V_h$ for every
$\sigma>0$.

\subsection{Interpolation, enrichment and embeddings}

\begin{lem}[\rm Interpolation \citep{brenner2012quadratic}]
\label{lem:P2interp}
\noindent Let $I_h:H^2_0(\O)\to V_h$ denote the nodal interpolation operator
associated with the $P_2$ Lagrange finite element space, prescribed for
$\varphi\in H^2_0(\O)$ by
\begin{align*}
(I_h\varphi)(z)=\varphi(z)\quad\text{ for any }z\in\cV_h\cup\cM_h .
\end{align*}
Then every $\varphi\in H^{2+\alpha}(\O)$ obeys the estimates $(i)$-$(iii)$.
\begin{itemize}
\item[(i)] $\|(1-I_h)\varphi\|_h\lesssim h^\alpha
|\varphi|_{H^{2+\alpha}(\O)},$
\item[(ii)] $\sum_{e\in\E_h^i}|e|^{-1}\|\sjump{\nabla I_h\varphi\cdot n_e}
\|^2_{L^2(e)}\lesssim h^{2\alpha}|\varphi|^2_{H^{2+\alpha}(\O)},$ and
\item[(iii)] $\sum_{e\in\E_h^i}|e|\|\sjump{D^2 I_h\varphi\,n_e}
\|^2_{L^2(e)}\lesssim h^{2\alpha}|\varphi|^2_{H^{2+\alpha}(\O)}$.
\end{itemize}
\end{lem}
\begin{proof}
Parts (ii) and (iii) are proved here; see also
\citep[Eq.~(2.13), Lem.~2.5]{brenner2012quadratic}. The gradient of $\varphi$
has no jump across an interior edge, so (i) gives
\[
\sum_{e\in\E_h^i}|e|^{-1}\|\sjump{\nabla I_h\varphi\cdot n_e}\|_{L^2(e)}^2
=\sum_{e\in\E_h^i}|e|^{-1}\|\sjump{\nabla((1-I_h)\varphi)\cdot n_e}
\|_{L^2(e)}^2
\lesssim\|(1-I_h)\varphi\|_h^2
\lesssim h^{2\alpha}|\varphi|_{H^{2+\alpha}(\O)}^2 .
\]
Let $e\in\E_h^i$ and let $Q_e$ be the union of the two triangles that share
$e$. The Bramble--Hilbert lemma provides $z_e\in\mathcal{P}_2(Q_e)$ with
$|\eta-z_e|_{H^2(Q_e)}\le C|e|^\alpha|\eta|_{H^{2+\alpha}(Q_e)}$. The Hessian
of $z_e$ has no jump across $e$. A triangle inequality and an inverse estimate
give
\begin{align*}
\sum_{e\in\E_h^i}|e|\|\sjump{D^2(I_h\varphi)\,n_e}\|_{L^2(e)}^2
&=\sum_{e\in\E_h^i}|e|\|\sjump{D^2(z_e-I_h\varphi)\,n_e}\|_{L^2(e)}^2
\le C\sum_{e\in\E_h^i}\sum_{T\in\T_e}|z_e-I_h\varphi|_{H^2(Q_e)}^2\\
&\le C\sum_{e\in\E_h^i}\sum_{T\in\T_e}
\bigl(|z_e-\varphi|_{H^2(Q_e)}^2+|(1-I_h)\varphi|_{H^2(Q_e)}^2\bigr)\\
&\lesssim h^{2\alpha}\sum_{e\in\E_h^i}\sum_{T\in\T_e}
|\varphi|_{H^{2+\alpha}(Q_e)}^2
\lesssim h^{2\alpha}|\varphi|_{H^{2+\alpha}(\O)}^2 .
\qedhere
\end{align*}
\end{proof}

\begin{lem}[enrichment \citep{brenner2012quadratic}]\label{lem:enrich}
Let $\widetilde{V}_h$ be the $P_6$ Argyris space on $\T_h$. There is a linear
operator $E_h:V_h\to\widetilde{V}_h\cap H^2_0(\O)$ such that every
$\varphi_h\in V_h$ satisfies
\begin{enumerate}
\item[\rm(i)] $\|(1-E_h)\varphi_h\|_{L^2(\O)}+
h|(1-E_h)\varphi_h|_{H^1(\O)}+h^2|E_h\varphi_h|_{H^2(\O)}
\lesssim h^2\|\varphi_h\|_h$;
\item[\rm(ii)] $\sum_{e\in\E_h^i}|e|^{-1}\|\smean{\p(
(1-E_h)\varphi_h)/\p n_e}\|_{L^2(e)}^2\lesssim\|\varphi_h\|_h^2$;
\item[\rm(iii)] $\sum_{e\in\E_h^i}|e|^{-1}\|
\p((1-E_h)\varphi_h)/\p t_e\|_{L^2(e)}^2\lesssim\|\varphi_h\|_h^2$;
\item[\rm(iv)] $(E_h\varphi_h)(p)=\varphi_h(p)$ for all $p\in\cV_h$;
\end{enumerate}
and every $\varphi\in H^{2+\alpha}(\O)\cap H^2_0(\O)$ satisfies
\begin{enumerate}
\item[\rm(v)] $\sum_{m=0}^{2}h^m|\varphi-E_hI_h\varphi|_{H^m(\O)}
\lesssim h^{2+\alpha}|\varphi|_{H^{2+\alpha}(\O)}$;
\item[\rm(vi)] $\sum_{e\in\E_h^i}|e|^{-1}
\|\p(\varphi-E_hI_h\varphi)/\p n_e\|_{L^2(e)}^2
\lesssim h^{2\alpha}|\varphi|_{H^{2+\alpha}(\O)}^2$;
\item[\rm(vii)] $\sum_{e\in\E_h^i}|e|^{-1}
\|\p(\varphi-E_hI_h\varphi)/\p t_e\|_{L^2(e)}^2
\lesssim h^{2\alpha}|\varphi|_{H^{2+\alpha}(\O)}^2$.
\end{enumerate}
\end{lem}

Property $\rm(iv)$ is what makes $E_h$ compatible with the obstacle
constraint: since $u_h\in K_h$ satisfies $u_h(p)\ge\chi(p)$ at every vertex,
$\rm(iv)$ gives $(E_hu_h)(p)=u_h(p)\ge\chi(p)$, so
\begin{align}\label{EhKh}
 E_hK_h\subset K_A
\end{align}
for the set $K_A$ of \Cref{sec:err-est}. This inclusion licenses the choice of
$E_hu_h$ as a test function in \eqref{c0ipwformA}.

\begin{lem}[\rm {Discrete embeddings}]\label{lem:emb-discrete}
\noindent It holds that
$$(i)\,\|\phi\|_{L^\infty(\O)}\lesssim\|\phi\|_h
\text{ for all }\phi\in H^2_0(\O)+V_h,\quad
(ii)\,\|\phi_h\|_{L^2(\O)}\lesssim\|\phi_h\|_h
\text{ for all }\phi_h\in V_h .$$
\end{lem}
\begin{proof}
The first estimate follows as in \citep[Lem.~3.7]{brenner2017c}. The second
follows from the first and from
$\|\bullet\|_{L^2(\O)}\lesssim\|\bullet\|_{L^\infty(\O)}$.
\end{proof}

The constants in \Cref{thm:disc-sob} enter the smallness condition of
\Cref{rem:disc-smallness}.

\begin{thm}[Discrete Sobolev and Friedrichs inequalities]
\label{thm:disc-sob}
\noindent Recall the elliptic regularity index $\alpha$. For $0<\alpha<1$, set
$\beta=\alpha$, and for $1\le\alpha$ and any $0<\epsilon<1$, set
$\beta=1-\epsilon$. Then there exist positive constants $C(\beta)$ and
$C(\alpha)$ for which the discrete embedding constants
$C_{\rm dS}:=C_{\rm S}+C(\beta)h_{\max}^{\beta}$ and
$C_{\rm dF}:=C_{\rm F}+C(\alpha)h_{\max}^{\alpha}$ satisfy $(a)$-$(b)$ for any
$v+v_h\in H^2_0(\O)+V_h$.
$$(a)\,\|v+v_h\|_{L^\infty(\O)}\le C_{\rm dS}\|v+v_h\|_h,
\qquad(b)\,\|v+v_h\|_{L^2(\O)}\le C_{\rm dF}\|v+v_h\|_h.$$
\end{thm}
\begin{proof}
The two inequalities follow as in
\citep[Lem.~4.7]{carstensen2017nonconforming}. The proof below gives the stated
form of the constants.

\emph{The first inequality.} Let $v\in H^2_0(\O)$ and $v_h\in V_h$. The
function $v+v_h$ is piecewise uniformly continuous, so there is a
$\varphi\in L^1(\O)$ with $\|\varphi\|_{L^1(\O)}=1$ and
$\|v+v_h\|_{L^\infty(\O)}=\int_\O(v+v_h)\varphi\,\dx$. Let $z\in H^2_0(\O)$
solve
\begin{equation}\label{z.solve}
a(z,\bullet)=\langle\varphi,\bullet\rangle_{L^1(\O)} .
\end{equation}
The embedding $H^{1+\epsilon}(\O)\hookrightarrow L^\infty(\O)$ gives
$\langle\varphi,\bullet\rangle_{L^1(\O)}\in H^{-(1+\epsilon)}(\O)$, and the
shift theorem \citep[Thm.~8]{bacuta2002shift} gives $z\in H^{2+\beta}(\O)$ with
\begin{equation}\label{sup.ds}
\|z\|_{2+\beta}\lesssim\|\langle\varphi,\bullet\rangle_{L^1(\O)}
\|_{H^{-(1+\epsilon)}}
=\sup_{0\neq\psi\in H^{1+\epsilon}_0(\O)}
\frac{(\varphi,\psi)_{L^2(\O)}}{\|\psi\|_{L^\infty(\O)}}
\frac{\|\psi\|_{L^\infty(\O)}}{\|\psi\|_{1+\epsilon}}
\le C(\beta,\O).
\end{equation}
Define
\[
a_1(\varphi_1,\varphi_2):=\sit D^2\varphi_1:D^2\varphi_2\,\dx
+\sum_{e\in\E_h^i}\int_e\smean{\p^2\varphi_1/\p n^2}
\sjump{\p\varphi_2/\p n}\,\ds ,
\]
let $w\in H^2_0(\O)$ solve $a(w,\bullet)=a_1(v+v_h,\bullet)$ in $H^{-2}(\O)$
and set $\delta:=v+v_h-w$, so that $a_1(\delta,\bullet)=0$ on $H^2_0(\O)$. Then
$\|\varphi\|_{L^1(\O)}=1$ and the Sobolev constant $C_{\rm S}$ give
\begin{align*}
(v+v_h,\varphi)_{L^2(\O)}
&=(w,\varphi)_{L^2(\O)}+(v+v_h-w,\varphi)_{L^2(\O)}
\le\|w\|_{L^\infty(\O)}+(v+v_h-w,\varphi)_{L^2(\O)}\\
&\le C_{\rm S}\trinl w\trinr+(\delta,\varphi)_{L^2(\O)}
\le C_{\rm S}\|v+v_h\|_h+(\delta,\varphi)_{L^2(\O)} .
\end{align*}
Since $E_h(I_h\delta)\in H^2_0(\O)$, the H\"older inequality with
\eqref{z.solve}, inverse estimates, a triangle inequality and the estimates of
\Cref{lem:P2interp} and \Cref{lem:enrich} give
\begin{align*}
(\delta,\varphi)_{L^2(\O)}
&=((1-I_h)\delta,\varphi)_{L^2(\O)}
+((1-E_h)I_h\delta,\varphi)_{L^2(\O)}
+(E_h(I_h\delta),\varphi)_{L^2(\O)}\\
&\le\|(1-I_h)\delta\|_{L^\infty(\O)}
+\|(1-E_h)I_h\delta\|_{L^\infty(\O)}+a(z,E_h(I_h\delta))\\
&\le\max_{T\in\T_h}\|(1-I_h)\delta\|_{L^\infty(T)}
+\max_{T\in\T_h}\|(1-E_h)I_h\delta\|_{L^\infty(T)}+a(z,E_h(I_h\delta))\\
&\lesssim h_{\max}\bigl(\|\delta\|_h+\|I_h\delta\|_h\bigr)
+a(z,E_h(I_h\delta))
\lesssim h_{\max}\|v+v_h\|_h+a(z,E_h(I_h\delta)) .
\end{align*}
The definition of $E_h$ and \Cref{lem:enrich-sol} give, for the last term,
\begin{align*}
a(z,E_h(I_h\delta))
&=a_1(z,(E_hI_h-1)\delta)
=a_1(z-I_hz,(E_hI_h-1)\delta)+a_1(I_hz,(E_hI_h-1)\delta)\\
&\lesssim\|z-I_hz\|_h\|(E_hI_h-1)\delta\|_h+a_1(I_hz,(E_hI_h-1)\delta)\\
&\lesssim h_{\max}^\alpha|z|_{2+\alpha}\|\delta\|_h
+a_1(I_hz,(E_hI_h-1)\delta)
\lesssim h_{\max}^\alpha|z|_{2+\alpha}\|v+v_h\|_h
+a_1(I_hz,(E_hI_h-1)\delta).
\end{align*}
Integration by parts and the Cauchy--Schwarz inequality applied to the
remaining term give
\begin{align*}
&a_1(I_hz,(E_hI_h-1)\delta)\\
&\quad=\sit D^2I_hz:D^2((E_hI_h-1)\delta)\,\dx
+\sum_{e\in\E_h^i}\int_e\smean{\frac{\p^2I_hz}{\p n^2}}
\sjump{\frac{\p((E_hI_h-1)\delta)}{\p n}}\,\ds\\
&\quad=-\sum_{e\in\E_h^i}\int_e\sjump{\frac{\p^2I_hz}{\p n^2}}
\smean{\frac{\p((E_hI_h-1)\delta)}{\p n}}\,\ds
-\sum_{e\in\E_h^i}\int_e\sjump{\frac{\p^2I_hz}{\p n\p t}}
\smean{\frac{\p((E_hI_h-1)\delta)}{\p t}}\,\ds\\
&\quad\lesssim\biggl(\sum_{e\in\E_h^i}|e|
\bigl\|\sjump{\p^2I_hz/\p n^2}\bigr\|_{L^2(e)}^2\biggr)^{1/2}
\biggl(\sum_{e\in\E_h^i}|e|^{-1}
\bigl\|\smean{\p((E_hI_h-1)\delta)/\p n}\bigr\|_{L^2(e)}^2\biggr)^{1/2}\\
&\qquad+\biggl(\sum_{e\in\E_h^i}|e|
\bigl\|\sjump{\p^2I_hz/\p n\p t}\bigr\|_{L^2(e)}^2\biggr)^{1/2}
\biggl(\sum_{e\in\E_h^i}|e|^{-1}
\bigl\|\p((E_hI_h-1)\delta)/\p t\bigr\|_{L^2(e)}^2\biggr)^{1/2}
=:AB+CD.
\end{align*}
\Cref{lem:P2interp} bounds the first factor,
\[
A^2=\sum_{e\in\E_h^i}|e|\bigl\|\sjump{\p^2I_hz/\p n^2}\bigr\|_{L^2(e)}^2
\lesssim h^{2\alpha}|z|_{H^{2+\alpha}(\O)}^2 ,
\]
and an inverse estimate followed by \Cref{lem:P2interp} bounds the third,
\[
C^2=\sum_{e\in\E_h^i}|e|\bigl\|\sjump{\p^2I_hz/\p n\p t}\bigr\|_{L^2(e)}^2
\lesssim\sum_{e\in\E_h^i}|e|^{-1}
\bigl\|\sjump{\p I_hz/\p n}\bigr\|_{L^2(e)}^2
\lesssim h^{2\alpha}|z|_{H^{2+\alpha}(\O)}^2 .
\]
\Cref{lem:enrich}, the triangle and trace inequalities and the estimates of
$\delta$ bound the two remaining factors,
\begin{align*}
B^2&\lesssim\sum_{e\in\E_h^i}|e|^{-1}
\bigl\|\smean{\p((E_h-1)I_h\delta)/\p n}\bigr\|_{L^2(e)}^2
+\sum_{e\in\E_h^i}|e|^{-1}
\bigl\|\smean{\p((I_h-1)\delta)/\p n}\bigr\|_{L^2(e)}^2\\
&\lesssim\|I_h\delta\|_h^2+\sum_{T\in\T_h}
\bigl(h_T^{-2}\|(I_h-1)\delta\|_{H^1(T)}^2
+\|(I_h-1)\delta\|_{H^2(T)}^2\bigr)
\lesssim\|\delta\|_h+\sum_{T\in\T_h}\|\delta\|_{H^2(T)}^2
\lesssim\|v+v_h\|_h,\\
D^2&\lesssim\sum_{e\in\E_h^i}|e|^{-1}
\bigl\|\p((E_h-1)I_h\delta)/\p t\bigr\|_{L^2(e)}^2
+\sum_{e\in\E_h^i}|e|^{-1}
\bigl\|\p((I_h-1)\delta)/\p t\bigr\|_{L^2(e)}^2\\
&\lesssim\|I_h\delta\|_h^2+\sum_{T\in\T_h}
\bigl(h_T^{-2}\|(I_h-1)\delta\|_{H^1(T)}^2
+\|(I_h-1)\delta\|_{H^2(T)}^2\bigr)
\lesssim\|\delta\|_h+\sum_{T\in\T_h}\|\delta\|_{H^2(T)}^2
\lesssim\|v+v_h\|_h .
\end{align*}
The bounds for $A$, $B$, $C$, $D$ and the preceding estimates give the first
inequality with $C_{\rm dS}=C_{\rm S}+C(\beta)h_{\max}^\beta$.

\emph{The second inequality.} Let $\varphi\in L^2(\O)$ with
$\|\varphi\|_{L^2(\O)}=1$ and let $z\in H^2_0(\O)$ solve
$a(z,\bullet)=(\varphi,\bullet)_{L^2(\O)}$. Then $z\in H^{2+\alpha}(\O)$ by
\citep[Thm.~2]{blum1980boundary}. The argument above applies with $L^\infty(\O)$
replaced by $L^2(\O)$, with $C_{\rm S}$ replaced by $C_{\rm F}$ and with
$\beta$ replaced by $\min\{\alpha,4\}$, and gives
$C_{\rm dF}=C_{\rm F}+C(\alpha)h_{\max}^\alpha$.
\end{proof}

\begin{lem}[\rm Enrichment \citep{brenner2012quadratic}]
\label{lem:enrich-sol}
\noindent Let $(u,v)$ be a solution to \eqref{wform}. Then $(i)$-$(iii)$ hold.
\begin{itemize}
\item[(i)] $|a(u,E_h(I_h-1)u)|\lesssim h^{2\alpha}$,
\item[(ii)] $|a(v,E_h(I_h-1)v)|\lesssim h^{2\alpha}$,
\item[(iii)] $|b(u,E_h(I_h-1)u,v)|\lesssim h^{2\alpha}$;
\end{itemize}
the suppressed constant in ``$\lesssim$'' is independent of the mesh size
parameter '$h$'.
\end{lem}

\begin{lem}\label{lem:trilinearprop}
\noindent Let $\psi,w\in H^{2+\alpha}(\O)\cap H^2_0(\O)$ and $\varphi_h\in V_h$. Then $(i)$-$(ii)$ hold.
\begin{itemize}
\item[(i)] $b_{\rm pw}(\psi,\varphi_h,w)
-b_{\rm j}(\psi,\varphi_h,w)
=\half\sit cof(D^2\psi)\nabla\varphi_h\cdot\nabla w\,\dx$, and
\item[(ii)] $b_{\rm j}(\varphi_h,\psi,w)=0$.
\end{itemize}
\end{lem}
\begin{proof}
The rows of $ cof(D^2\psi)$ are divergence free and $ cof(D^2\psi)$ is continuous. Integration by parts on each triangle gives (i); see \citep[Lem.~4.2 \& Eq.~(4.3)]{brenner2017c}. Since $\psi\in H^{2+\alpha}(\O)\subset H^2(\O)$, the gradient of $\psi$ has no jump
across an interior edge and vanishes on 
$\partial\O$, and (ii) follows from the definition of
$b_{\rm j}(\bullet,\bullet,\bullet)$.
\end{proof}

\begin{lem}[\rm Bounds for $b_{\rm IP}(\bullet,\bullet,\bullet)$
{\citep[Lemma 2.6]{carstensen2019adaptive}}]\label{lem:bIPbound}
\noindent For all $\eta_{\rm u},\varphi_h,\phi_h\in V_h$,
$$b_{\rm IP}(\eta_{\rm u},\varphi_h,\phi_h)\le
C_{\rm dS}\|\eta_{\rm u}\|_h\|\varphi_h\|_h
\|\phi_h\|_h.$$
\end{lem}

\section{The discrete problem}\label{sec:disc-prob}

\subsection{Formulation}

The discrete admissible set is
\[
K_h:=\bigl\{\varphi_h\in V_h\,:\,\chi(p)\le\varphi_h(p)\fl
p\in\cV_h\bigr\} .
\]
It is a non-empty closed convex subset of $V_h$. A function in $K_h$ may fall
below the obstacle in the interior of a triangle, so $K_h$ is not a subset of
$K$. The $C^0$ interior penalty method seeks
$(u_h,v_h)\in K_h\times V_h$ with
\begin{subequations}\label{C0IPwform}
\begin{align}
a_{\rm IP}(u_h,u_h-\varphi_1)+2b_{\rm IP}(u_h,u_h-\varphi_1,v_h)
&\le(f,u_h-\varphi_1)_{L^2(\O)}\fl\varphi_1\in K_h,
\label{C0IPwforma}\\
a_{\rm IP}(v_h,\varphi_2)-b_{\rm IP}(u_h,u_h,\varphi_2)
&=0\fl\varphi_2\in V_h.\label{C0IPwformb}
\end{align}
\end{subequations}
The form $b_{\rm IP}(\bullet,\bullet,\bullet)$ is symmetric in its first two
arguments only. The position of the test function in \eqref{C0IPwforma} and in
\eqref{C0IPwformb} is therefore part of the definition of the method.

\subsection{Existence, {\em a priori} bound and uniqueness}

Let $\sigma$ be large enough for $a_{\rm IP}(\bullet,\bullet)$ to be coercive
on $V_h$; the value $\sigma=5$ suffices \citep[Rem.~2.3]{brenner2012quadratic}. Define $G_h:V_h\to V_h$ by
\[
a_{\rm IP}(G_h(\phi_h),\psi_h)=b_{\rm IP}(\phi_h,\phi_h,\psi_h)
\fl\phi_h,\psi_h\in V_h
\]
and the discrete energy
\[
j_h(\phi_h):=\tfrac12 a_{\rm IP}(\phi_h,\phi_h)
+\tfrac12 a_{\rm IP}(G_h(\phi_h),G_h(\phi_h))
-(f,\phi_h)_{L^2(\O)} .
\]
{\bf Note:} By applying the Fr\'echet derivative and the symmetry of
$b_{\rm IP}(\bullet,\bullet,\bullet)$ in its first two arguments, one verifies
that any minimiser of the energy functional $j_h(\bullet)$ over $K_h$ is a
solution to \eqref{C0IPwform}.

\begin{thm}[Existence, {\em a priori} bound and uniqueness condition]
\label{thm:disc-existence}
\noindent Let $(f,\chi)\in L^2(\O)\times H^2(\O)$ satisfy
$\max\chi(\partial\O)<0$. Then the functional $j_h(\bullet)$ attains its
minimum over $K_h$ at some $u_h\in K_h$, and this minimiser together with
$v_h:=G_h(u_h)$ solves \eqref{C0IPwform}. Moreover, there is a
positive constant $C_{\rm d}(\chi)$ that depends only on $\chi$ such that any
solution $(u_h,v_h)$ to \eqref{C0IPwform} satisfies $(a)$-$(b)$.\\
$(a)$ $\frac{1}{2}\|u_h\|_h^2+\|v_h\|_h^2
\le N_{\rm d}^2(f,u_h):=2j_h(u_h)
+2C^2_{\rm dF}\|f\|_{L^2(\O)}^2$

$\qquad\qquad\qquad\qquad\,
\le M_{\rm d}^2(f,\chi):=C_{\rm d}(\chi)+3C^2_{\rm dF}\|f\|^2_{L^2(\O)}.$\\
$(b)$ If $\frac{C_{\rm dS}^2}{4}\|u_h\|_h^2
+C_{\rm dS}\|v_h\|_h<\half$, then $(u_h,v_h)$ is the
only solution to \eqref{C0IPwform}.
\end{thm}
\begin{proof}
The set $K_h$ is non-empty, closed and convex in the finite-dimensional space
$V_h$, and $j_h$ is continuous and coercive on $K_h$ by the lower bound in (a).
A minimiser therefore exists. The remaining assertions follow as in
\citep[Thm.~4.6]{carstensen2021morley} with $\trinl\bullet\trinr_{\pw}$ replaced
by $\|\bullet\|_h$ and with $C_{\rm S}$ and $C_{\rm F}$ replaced by
$C_{\rm dS}$ and $C_{\rm dF}$ from \Cref{thm:disc-sob}.
\end{proof}

\begin{rem}[{\em a priori} and {\em a posteriori} criteria for discrete
uniqueness]\label{rem:disc-smallness}
The {\em a priori} smallness assumption on the data
$C_{\rm dS}M_{\rm d}(f,\chi)<1/2$ implies the {\em a posteriori} smallness
assumption $C_{\rm dS}N_{\rm d}(f,u_h)<1/2$ and so the uniqueness of the
solution to \eqref{C0IPwform}. Since \Cref{thm:disc-sob} provides
$C_{\rm dS}\to C_{\rm S}$ and $C_{\rm dF}\to C_{\rm F}$ as $h_{\max}\to0$, the
continuous condition $C_{\rm S}M(f,\chi)<1/2$ of \Cref{thm:contsdep} implies
its discrete counterpart for all sufficiently small $h_{\max}$.
\end{rem}

\section{Error analysis}\label{sec:err-est}

\subsection{Two auxiliary problems}

Let $(u,v)$ solve \eqref{wform}. \Cref{thm:reg} gives
$u,v\in H^{2+\alpha}(\O)$, and the embedding
$H^{2+\alpha}(\O)\hookrightarrow W^{2,4}(\O)$ gives
$\widetilde f:=f+[u,v]\in L^2(\O)$. The first auxiliary problem seeks
$u_L\in K$ with
\begin{equation}\label{coipwformT}
a(u_L,u_L-\varphi)\le(\widetilde f,u_L-\varphi)_{L^2(\O)}
\fl\varphi\in K ,
\end{equation}
equivalently $u_L=\argmin_{\xi\in K}J_T(\xi)$ with
$J_T(\xi):=\tfrac12 a(\xi,\xi)-(\widetilde f,\xi)_{L^2(\O)}$. The solution $u$
of \eqref{wform} solves \eqref{coipwformT}, so $u_L=u$. The second auxiliary
problem replaces $K$ by
$K_A:=\{\xi\in H^2_0(\O):\xi(p)\ge\chi(p)\fl p\in\cV_h\}$ and seeks
$u_A\in K_A$ with
\begin{equation}\label{c0ipwformA}
a(u_A,u_A-\varphi)\le(\widetilde f,u_A-\varphi)_{L^2(\O)}
\fl\varphi\in K_A ,
\end{equation}
equivalently $u_A=\argmin_{\xi\in K_A}J_T(\xi)$.

\begin{lem}[\rm{Convergence rates} \citep{brenner2012finite}]
\label{lem:auxiliary-rates}
\noindent Let $\chi\in C^2(\O)$ with $\max_{x\in\partial\O}\chi(x)<0$, let
$u_L\in C^2(\O)\cap H^{2+\alpha}(\O)$ be a solution to \eqref{coipwformT} and
let $u_A$ be a solution to \eqref{c0ipwformA}. Then there exist an $h_0>0$ and
a $\widehat u_A\in K$ such that
$$\trinl u_L-u_A\trinr\lesssim h\quad\text{and}\quad
\trinl\widehat u_A-u_A\trinr\lesssim h^2$$
for any $h<h_0$.
\end{lem}

\subsection{The discrete energy norm estimate}

\begin{thm}[energy norm estimate]\label{thm:err}
Let $f\in L^2(\O)$ and let $\chi\in C^2(\O)$ with
$\max_{x\in\partial\O}\chi(x)<0$ satisfy $C_{\rm S}M(f,\chi)<1/2$, and let
$(u,v)$ be the unique solution of \eqref{wform}. Then there exist $h_0>0$ and
$C>0$ that are independent of the mesh size such that, for every triangulation
with $h_{\max}<h_0$ and every sufficiently large $\sigma$, the problem
\eqref{C0IPwform} has a unique solution $(u_h,v_h)$ and
\[
\|u-u_h\|_h+\|v-v_h\|_h\le C\,h^{\alpha}.
\]
\end{thm}

\begin{proof}
The proof extends \citep[Thm.~5.1]{carstensen2021morley} to the penalty setting
and has four steps. Set $\eta_u:=I_hu-u_h$ and $\eta_v:=I_hv-v_h$.

\emph{Step 1.} Coercivity of $a_{\rm IP}(\bullet,\bullet)$ and
\eqref{C0IPwform} give
\begin{align}
\|\eta_u\|_h^2&\le a_{\rm IP}(\eta_u,\eta_u)
\le a_{\rm IP}(I_hu,\eta_u)
-(f,\eta_u)_{L^2(\O)}+2b_{\rm IP}(u_h,\eta_u,v_h),
\label{c0ipeqn1}\\
\|\eta_v\|_h^2&\le a_{\rm IP}(\eta_v,\eta_v)
\le a_{\rm IP}(I_hv,\eta_v)-b_{\rm IP}(u_h,u_h,\eta_v).
\label{c0ipeqn2}
\end{align}

\emph{Step 2.} Recall $a_{\rm IP}(\bullet,\bullet)=a_{\rm pw}(\bullet,\bullet)
+a_{\rm J}(\bullet,\bullet)$. Insert $E_h\eta_u\in H^2_0(\O)$ in the first term
on the right-hand side of \eqref{c0ipeqn1},
\begin{align}\label{c0ipeqn3u}
a_{\rm IP}(I_hu,I_hu-u_h)=a_{\rm pw}(I_hu,(1-E_h)\eta_u)
+a_{\rm pw}(I_hu,E_h\eta_u)+a_{\rm J}(I_hu,\eta_u),
\end{align}
and $E_h\eta_v$ in the first term on the right-hand side of \eqref{c0ipeqn2},
\begin{align}\label{c0ipeqn3v}
a_{\rm IP}(I_hv,I_hv-v_h)=a_{\rm pw}(I_hv,(1-E_h)\eta_v)
+a_{\rm pw}(I_hv,E_h\eta_v)+a_{\rm J}(I_hv,\eta_v).
\end{align}
Combine the first and the last term on the right-hand side of
\eqref{c0ipeqn3u} and integrate by parts to obtain
\begin{align*}
&a_{\rm pw}(I_hu,(1-E_h)\eta_u)
+\sum_{e\in\E_h^i}\int_e\smean{\frac{\p^2I_hu}{\p n^2}}
\sjump{\frac{\p\eta_u}{\p n}}\,\ds\\
&\qquad=\sum_{e\in\E_h^i}\int_e\sjump{\frac{\p^2I_hu}{\p n^2}}
\smean{\frac{\p((1-E_h)\eta_u)}{\p n}}\,\ds
+\sum_{e\in\E_h^i}\int_e\sjump{\frac{\p^2I_hu}{\p n\p t}}
\smean{\frac{\p((1-E_h)\eta_u)}{\p t}}\,\ds .
\end{align*}
The Cauchy--Schwarz inequality, standard inverse estimates and the estimates
of \Cref{lem:P2interp} and \Cref{lem:enrich} give
\begin{align}
&a_{\rm pw}(I_hu,(1-E_h)\eta_u)
+\sum_{e\in\E_h^i}\int_e\smean{\frac{\p^2I_hu}{\p n^2}}
\sjump{\frac{\p\eta_u}{\p n}}\,\ds\nonumber\\
&\qquad\lesssim\biggl(\sum_{e\in\E_h^i}|e|
\bigl\|\sjump{\p^2I_hu/\p n^2}\bigr\|_{L^2(e)}^2\biggr)^{1/2}
\biggl(\sum_{e\in\E_h^i}|e|^{-1}
\bigl\|\smean{\p((1-E_h)\eta_u)/\p n}\bigr\|_{L^2(e)}^2\biggr)^{1/2}
\nonumber\\
&\qquad\quad+\biggl(\sum_{e\in\E_h^i}|e|^{-1}
\bigl\|\sjump{\p I_hu/\p n}\bigr\|_{L^2(e)}^2\biggr)^{1/2}
\biggl(\sum_{e\in\E_h^i}|e|^{-1}
\bigl\|\smean{\p((1-E_h)\eta_u)/\p t}\bigr\|_{L^2(e)}^2\biggr)^{1/2}
\nonumber\\
&\qquad\lesssim h^{\alpha}\|\eta_u\|_h ,
\end{align}
and the same argument for $v$ gives
\begin{align}
a_{\rm pw}(I_hv,(1-E_h)\eta_v)
+\sum_{e\in\E_h^i}\int_e\smean{\frac{\p^2I_hv}{\p n^2}}
\sjump{\frac{\p\eta_v}{\p n}}\,\ds\lesssim h^{\alpha}\|\eta_v\|_h .
\end{align}
For the remaining terms of $a_{\rm J}(\bullet,\bullet)$ use
$\sjump{\p u/\p n}=0$ on the interior edges, the Cauchy--Schwarz inequality
and \Cref{lem:P2interp},
\begin{align}
&\sum_{e\in\E_h^i}\int_e\smean{\frac{\p^2\eta_u}{\p n^2}}
\sjump{\frac{\p I_hu}{\p n}}\,\ds
+\sigma\sum_{e\in\E_h^i}|e|^{-1}\int_e
\sjump{\frac{\p I_hu}{\p n}}\sjump{\frac{\p\eta_u}{\p n}}\,\ds\nonumber\\
&\quad=\sum_{e\in\E_h^i}\int_e\smean{\frac{\p^2\eta_u}{\p n^2}}
\sjump{\frac{\p((I_h-1)u)}{\p n}}\,\ds
+\sigma\sum_{e\in\E_h^i}|e|^{-1}\int_e
\sjump{\frac{\p((I_h-1)u)}{\p n}}
\sjump{\frac{\p\eta_u}{\p n}}\,\ds\nonumber\\
&\quad\lesssim\biggl(\sum_{e\in\E_h}\Bigl(|e|
\bigl\|\smean{\p^2\eta_u/\p n^2}\bigr\|_{L^2(e)}^2
+|e|^{-1}\bigl\|\sjump{\p\eta_u/\p n}\bigr\|_{L^2(e)}^2
\Bigr)\biggr)^{1/2}\nonumber\\
&\qquad\qquad\times\biggl(\sum_{e\in\E_h}|e|^{-1}
\bigl\|\sjump{\p((I_h-1)u)/\p n}\bigr\|_{L^2(e)}^2\biggr)^{1/2}
\lesssim h^{\alpha}\|\eta_u\|_h ,
\end{align}
and the analogous estimate holds with $u$ and $\eta_u$ replaced by $v$ and
$\eta_v$.

\emph{Step 3.} Insert $u$, $\widehat u_A$, $u_A$ and $E_hI_hu$ in the middle term of \eqref{c0ipeqn3u}. Boundedness of $a_{\rm pw}(\bullet,\bullet)$, the weak formulations \eqref{wform} and \eqref{c0ipwformA},
\Cref{lem:enrich-sol}, \Cref{lem:enrich} and \Cref{lem:auxiliary-rates} give
\begin{align}\label{c0ipest4}
a_{\rm pw}(I_hu,E_h\eta_u)
&=a_{\rm pw}((I_h-1)u,E_h\eta_u)+a_{\rm pw}(u,(E_hI_h-1)u)
+a_{\rm pw}(u,u-\widehat u_A)\nonumber\\
&\qquad+a_{\rm pw}(u,\widehat u_A-u_A)
+a_{\rm pw}(u-u_A,u_A-E_hu_h)+a_{\rm pw}(u_A,u_A-E_hu_h)\nonumber\\
&\lesssim h^{\alpha}\|\eta_u\|_h+h^{2\alpha}+a_{\rm pw}(u,u-\widehat u_A)
+h^2+h\|\eta_u\|_h+a_{\rm pw}(u_A,u_A-E_hu_h).
\end{align}
Combine the remaining terms of \eqref{c0ipest4} with the last two terms of
\eqref{c0ipeqn1}, use the weak formulations \eqref{wform} and
\eqref{c0ipwformA} with the admissible test function $E_hu_h\in K_A$ from
\eqref{EhKh}, \Cref{lem:enrich-sol}, \Cref{lem:enrich} and
\Cref{lem:auxiliary-rates} to obtain,
\begin{align}\label{c0ipest6}
&a_{\rm pw}(u,u-\widehat u_A)+a_{\rm pw}(u_A,u_A-E_hu_h)
-(f,\eta_u)_{L^2(\O)}+2b_{\rm IP}(u_h,I_hu-u_h,v_h)\nonumber\\
&\quad\le(f+[u,v],u-\widehat u_A)_{L^2(\O)}
+(f+[u,v],u_A-E_hu_h)_{L^2(\O)}
-(f,\eta_u)_{L^2(\O)}+2b_{\rm IP}(u_h,I_hu-u_h,v_h)\nonumber\\
&\quad\le(f+[u,v],u_A-\widehat u_A)_{L^2(\O)}
+\bigl(f,(u-E_hu_h)-I_h(u-E_hu_h)\bigr)_{L^2(\O)}\nonumber\\
&\qquad\qquad+([u,v],u-E_hu_h)_{L^2(\O)}
+2b_{\rm IP}(u_h,I_hu-u_h,v_h)\nonumber\\
&\quad\lesssim h^2+h^{2+\alpha}+h^2\|\eta_u\|_h
-2b(u,u-E_hu_h,v)+2b_{\rm IP}(u_h,I_hu-u_h,v_h).
\end{align}
For the middle term of \eqref{c0ipeqn3v} insert $v$ and use boundedness of
$a_{\rm pw}(\bullet,\bullet)$, \eqref{wformb}, \Cref{lem:enrich-sol} and
\Cref{lem:enrich},
\begin{align}\label{c0ipest5}
a_{\rm pw}(I_hv,E_h\eta_v)
&=a_{\rm pw}((I_h-1)v,E_h\eta_v)+a_{\rm pw}(v,E_hI_hv-v)
+a_{\rm pw}(v,v-E_hv_h)\nonumber\\
&\lesssim h^{\alpha}\|\eta_v\|_h+h^{2\alpha}+b(u,u,v-E_hv_h).
\end{align}

\emph{Step 4.} Collect the terms with $b_{\rm IP}(\bullet,\bullet,\bullet)$
and $b(\bullet,\bullet,\bullet)$ from \eqref{c0ipeqn2}, \eqref{c0ipest6} and
\eqref{c0ipest5}. Elementary manipulations give
\begin{align*}
&-b(u,u-E_hu_h,v)+b_{\rm IP}(u_h,I_hu-u_h,v_h)
+b(u,u,v-E_hv_h)-b_{\rm IP}(u_h,u_h,I_hv-v_h)\\
&\qquad=-b_{\rm IP}(I_hu-u_h,I_hu-u_h,v_h)
+b_{\rm IP}(I_hu-u_h,u_h,I_hv-v_h)
-b_{\rm IP}((1-I_h)u,I_hu-u_h,v_h)\\
&\qquad\quad+b_{\rm IP}((1-I_h)u,u_h,I_hv-v_h)
+b_{\rm IP}(u,(1-I_h)u,v-v_h)
-b_{\rm IP}(u,u-u_h,(1-I_h)v)\\
&\qquad\quad+b_{\rm IP}(u,(1-E_h)(I_hu-u_h),v)
+b(u,u,(1-E_h)(v_h-I_hv))
-b(u,(1-E_hI_h)u,v)\\
&\qquad\quad+b(u,u,(1-E_hI_h)v).
\end{align*}
\Cref{lem:trilinearprop}, \Cref{lem:bIPbound}, boundedness of
$b(\bullet,\bullet,\bullet)$, \citep[Lem.~2.2]{brenner2017c} and a triangle
inequality lead to
\begin{align}\label{c0ipestbs}
&-b(u,u-E_hu_h,v)+b_{\rm IP}(u_h,I_hu-u_h,v_h)
+b(u,u,v-E_hv_h)-b_{\rm IP}(u_h,u_h,I_hv-v_h)\nonumber\\
&\quad\le C_{\rm dS}\|v_h\|_h\|I_hu-u_h\|_h^2
+C_{\rm dS}\|u_h\|_h\|I_hu-u_h\|_h\|I_hv-v_h\|_h
+C_{\rm dS}\|v_h\|_h\|(1-I_h)u\|_h\|u-u_h\|_h\nonumber\\
&\qquad+C_{\rm dS}(\|u_h\|_h+\trinl u\trinr)\|(1-I_h)u\|_h\|I_hv-v_h\|_h
\nonumber\\
&\qquad+C_{\rm dS}\trinl u\trinr
\bigl(2\|(1-I_h)u\|_h+\|I_hu-u_h\|_h\bigr)\|(1-I_h)v\|_h\nonumber\\
&\qquad+\|D^2u\|_{L^4(\O)}\|\nabla v\|_{L^4(\O)}
\|\nabla(1-E_h)(I_hu-u_h)\|
+\trinl u\trinr^2\|(1-E_h)(v_h-I_hv)\|_{L^\infty(\O)}\nonumber\\
&\qquad+\|D^2u\|_{L^4(\O)}\|\nabla v\|_{L^4(\O)}\|\nabla((1-E_hI_h)u)\|
+\|D^2u\|_{L^4(\O)}\|\nabla u\|_{L^4(\O)}\|\nabla((1-E_hI_h)v)\| .
\end{align}
Combine \eqref{c0ipeqn1}--\eqref{c0ipestbs}, apply a triangle inequality and
the estimates of \Cref{lem:P2interp} and \Cref{lem:enrich} to obtain, for every
$\gamma\in(0,1]$,
\begin{align}\label{DiscEnergyEst}
\Bigl(1-\tfrac{1}{2\gamma}C_{\rm dS}^2\|u_h\|_h^2
-2C_{\rm dS}\|v_h\|_h\Bigr)&\|I_hu-u_h\|_h^2
+2(1-\gamma)\|I_hv-v_h\|_h^2\nonumber\\
&\lesssim\mathcal{O}(h^{2\alpha})
+\mathcal{O}(h^{\alpha})\bigl(\|I_hu-u_h\|_h+\|I_hv-v_h\|_h\bigr).
\end{align}
Set $g(\gamma):=\tfrac{1}{2\gamma}C_{\rm dS}^2\|u_h\|_h^2
+2C_{\rm dS}\|v_h\|_h$. \Cref{thm:disc-existence} gives $g(1)<\tfrac12$, and
$g$ is continuous, so there is a $\gamma_0$ with $0<\gamma_0<1$ and
$g(\gamma_0)<\tfrac12$. This and a sufficiently large penalty parameter
$\sigma$ in $a_{\rm J}(\bullet,\bullet)$ bound the left-hand side of
\eqref{DiscEnergyEst} from below,
\begin{align}\label{CoercivitySmallBilinear}
&\|I_hu-u_h\|_h^2+\|I_hv-v_h\|_h^2\nonumber\\
&\qquad\lesssim(1-g(\gamma_0))\,a_{\rm pw}(I_hu-u_h,I_hu-u_h)
+a_{\rm J}(I_hu-u_h,I_hu-u_h)\nonumber\\
&\qquad\quad+2(1-\gamma_0)\,a_{\rm pw}(I_hv-v_h,I_hv-v_h)
+a_{\rm J}(I_hv-v_h,I_hv-v_h).
\end{align}
Combine \eqref{DiscEnergyEst} and \eqref{CoercivitySmallBilinear}, let $C$ be
the constant suppressed in $\lesssim$, which depends on $\trinl u\trinr$,
$\trinl v\trinr$, $\|u\|_{2+\alpha}$, $\|v\|_{2+\alpha}$ and
$\|f\|_{L^2(\O)}$, and apply the Young inequality to obtain
\begin{equation}\label{c0ipestfinal}
\|I_hu-u_h\|_h^2+\|I_hv-v_h\|_h^2\le 2C(1+C)h^{2\alpha}.
\end{equation}
The triangle inequality and \Cref{lem:P2interp} complete the proof.
\end{proof}

\begin{rem}[the three error contributions]\label{rem:rate}
Three quantities enter \eqref{c0ipestfinal}. The interpolation error of $P_2$
elements in $\|\bullet\|_h$ is $\mathcal{O}(h^\alpha)$ by \Cref{lem:P2interp}.
The replacement of the constraint in $K$ by the constraint at the vertices is
$\mathcal{O}(h)$ by \Cref{lem:auxiliary-rates}. The nonconformity of $V_h$ in
$H^2_0(\O)$ is $\mathcal{O}(h^\alpha)$ by \Cref{lem:enrich}. On a convex
polygon $\alpha=1$ and the three contributions have the same order. On a domain
with a reentrant corner the first and the third have order $\alpha<1$ and
dominate.
\end{rem}

\section{Numerical experiments}\label{sec:numerics}

\noindent This section reports the numerical realisation of \eqref{C0IPwform} and four
experiments. \Cref{ssec:solver} describes the algebraic form of the discrete
problem and the iterative solver. \Cref{ssec:errors} defines the error
quantities. \Cref{ssec:square} and \Cref{ssec:lshape} present the convergence
histories on a square and on an L-shaped domain, \Cref{ssec:sigma} the
dependence of the observed rates on the penalty parameter $\sigma$, and
\Cref{ssec:smallness} the failure of the smallness condition of
\Cref{rem:disc-smallness} for large data.

\subsection{Implementation}\label{ssec:solver}\label{ssec:algebraic}

Let $\varphi_1,\dots,\varphi_N$ be the Lagrange basis of $V_h$ associated with
the interior vertices and the midpoints of the interior edges, and let
$u_h=\sum_j\alpha_j\varphi_j$ and $v_h=\sum_j\beta_j\varphi_j$. Set
\[
 A_{kj}:=a_{\rm IP}(\varphi_j,\varphi_k),\qquad
 F_k:=(f,\varphi_k)_{L^2(\O)},\qquad
 T_{ijk}:=b_{\rm IP}(\varphi_i,\varphi_j,\varphi_k).
\]
The definition \eqref{defn-bh} and the reduction of the edge terms in
\Cref{sec:disc-setup} give
\begin{align}\label{tensorT}
 T_{ijk}=-\frac12\sit[\varphi_i,\varphi_j]\varphi_k\,\dx
 +\frac12\sum_{e\in\E_h^i}\int_e\Bigl(
 \smean{\frac{\p^2\varphi_i}{\p t^2}}\sjump{\frac{\p\varphi_j}{\p n}}
 +\smean{\frac{\p^2\varphi_j}{\p t^2}}\sjump{\frac{\p\varphi_i}{\p n}}
 \Bigr)\varphi_k\,\ds ,
\end{align}
where the boundary edges contribute nothing because $\varphi_k$ vanishes on
$\partial\O$, and $T_{ijk}=T_{jik}$. With the obstacle values
$\chi_p:=\chi(p)$ at the interior vertices, \eqref{C0IPwform} is the
complementarity system
\begin{subequations}\label{algsystem}
\begin{align}
 (A\bdalpha)_k+2\sum_{i,j}T_{ikj}\,\alpha_i\beta_j-F_k&=\lambda_k,
 \label{algsystema}\\
 (A\bdbeta)_k-\sum_{i,j}T_{ijk}\,\alpha_i\alpha_j&=0,
 \label{algsystemb}\\
 \alpha_p\ge\chi_p,\quad\lambda_p\ge0,\quad
 \lambda_p(\alpha_p-\chi_p)&=0\fl p\in\cV_h(\O),
 \label{algsystemc}
\end{align}
\end{subequations}
with $\lambda_p=0$ at the edge midpoints, since the constraint in $K_h$ acts at
the vertices. The two equations contract $T$ over different pairs of indices,
because $b_{\rm IP}$ is symmetric in its first two arguments only.

The system \eqref{algsystem} is solved by the primal--dual active set strategy
of \citep{hintermuller2002primaldual} with a Newton iteration in the inner loop,
as in \citep[Sec.~6.1]{carstensen2021morley}; the linear case, in which the
inner loop is absent, is treated in \citep{brenner2012quadratic}. In the $m$-th
outer step the active set
\[
 \cA^m:=\{p\in\cV_h(\O)\,:\,\lambda^{m-1}_p+c\,(\chi_p-\alpha^{m-1}_p)>0\},
 \qquad c=1,
\]
fixes $\alpha_p=\chi_p$ on $\cA^m$ and $\lambda_p=0$ elsewhere, and the
resulting smooth system is solved by Newton's method. The discrete biharmonic
obstacle problem, obtained by deleting the terms with $T$ in
\eqref{algsystem}, provides the initial iterate together with $\bdbeta=0$. The
outer loop stops when the active set repeats, which is the criterion recorded
in \Cref{ssec:smallness}. The computations use at most ten Newton steps per
active set with tolerance $10^{-7}$ on the increment and at most forty outer
iterations.

\subsection{Error quantities}\label{ssec:errors}

Each triangulation after the first is obtained by red refinement, so
$h_\ell=h_1 2^{1-\ell}$ with $h_1$ stated for each example. The exact solution
of \eqref{wform} is not known for these obstacles, so the discrete solution on
the finest level $L$ replaces it, as in \citep{carstensen2021morley}. Let
$u_\ell$ and $v_\ell$ be the discrete solutions on the $\ell$-th triangulation
and set
\begin{align}\label{errdefs}
 e_\ell(u):=\|u_L-u_\ell\|_h,
 \qquad
 \widetilde e_\ell(u):=\max_{p\in\cV_\ell}|u_L(p)-u_\ell(p)| ,
\end{align}
and likewise for $v$. The $P_2$ spaces are nested under red refinement, so
$u_\ell$ is represented exactly on $\T_L$. The empirical order of convergence
is
\begin{align}\label{eocdef}
 {\rm EOC}(\ell):=
 \frac{\log\bigl(e_\ell(u)/e_{L-1}(u)\bigr)}{\log(2^{\,L-1-\ell})},
 \qquad \ell=1,\dots,L-2,
\end{align}
with the same expression for $\widetilde e_\ell$, and the discrete coincidence
set on level $\ell$ is
$\cC_\ell:=\{p\in\cV_\ell\,:\,u_\ell(p)-\chi(p)\le\widetilde e_\ell(u)\}$.
\subsection{The obstacle problem on the square}\label{ssec:square}

Let $\O=(-\tfrac12,\tfrac12)^2$ with initial triangulation the criss-cross mesh
through the four corners and the centre, so $h_1=1$, and let $f\equiv0$. The two
obstacles are those of \citep{brenner2013morley} for the biharmonic plate, taken
up for the Morley discretisation of the von K\'arm\'an obstacle problem in \citep{carstensen2021morley}. All tables in this subsection and in \Cref{ssec:lshape} use $\sigma=10$; \Cref{ssec:sigma} explains the choice.

\begin{example}[contact set of positive measure]\label{ex:sq1}
Let $\chi(x)=1-5|x|^2+|x|^4$, for which $\Delta^2\chi=64>0$. The complement
$\O\setminus\cC$ of the coincidence set is then connected by the analysis of
\citep{caffarelli1979obstacle}, and the computed sets in
\Cref{fig:ex1} are solid regions whose complements are connected.
\end{example}

\begin{table}[htbp]
\centering
\footnotesize
\setlength{\tabcolsep}{3.5pt}
\caption{Convergence history for \Cref{ex:sq1} at $\sigma=10$. The seven levels carry $5$, $25$, $113$, $481$, $1985$, $8065$ and $32513$ degrees of freedom; the finest serves as reference.}
\label{tab:ex1}
\begin{tabular}{||c|c||r|c||r|c||r|c||r|c||}
\hline\hline
$\ell$ & $h_\ell$ & $\widetilde e_\ell(u)$ & EOC & $\widetilde e_\ell(v)$ & EOC & $e_\ell(u)$ & EOC & $e_\ell(v)$ & EOC \\
\hline\hline
1 & 1.0000 & 0.000000 & -- & 0.108142 & 1.8397 & 162.058417 & 1.2328 & 5.206544 & 1.0946 \\
2 & 0.5000 & 0.013239 & 1.1021 & 0.040385 & 1.9443 & 80.609565 & 1.2892 & 3.515582 & 1.2266 \\
3 & 0.2500 & 0.042209 & 2.0270 & 0.008111 & 1.8204 & 38.115105 & 1.3587 & 1.758439 & 1.3023 \\
4 & 0.1250 & 0.008553 & 1.8890 & 0.001549 & 1.5363 & 16.794320 & 1.4469 & 0.766334 & 1.3543 \\
5 & 0.0625 & 0.001318 & 1.0797 & 0.000915 & 2.3133 & 6.117536 & 1.4368 & 0.310426 & 1.4048 \\
6 & 0.0312 & 0.000624 & -- & 0.000184 & -- & 2.259801 & -- & 0.117236 & -- \\
\hline\hline
\end{tabular}
\end{table}

\begin{figure}[htbp]
\centering
\subfloat[$u_h$ and $\chi$, level $7$]{%
    \includegraphics[width=0.60\textwidth]{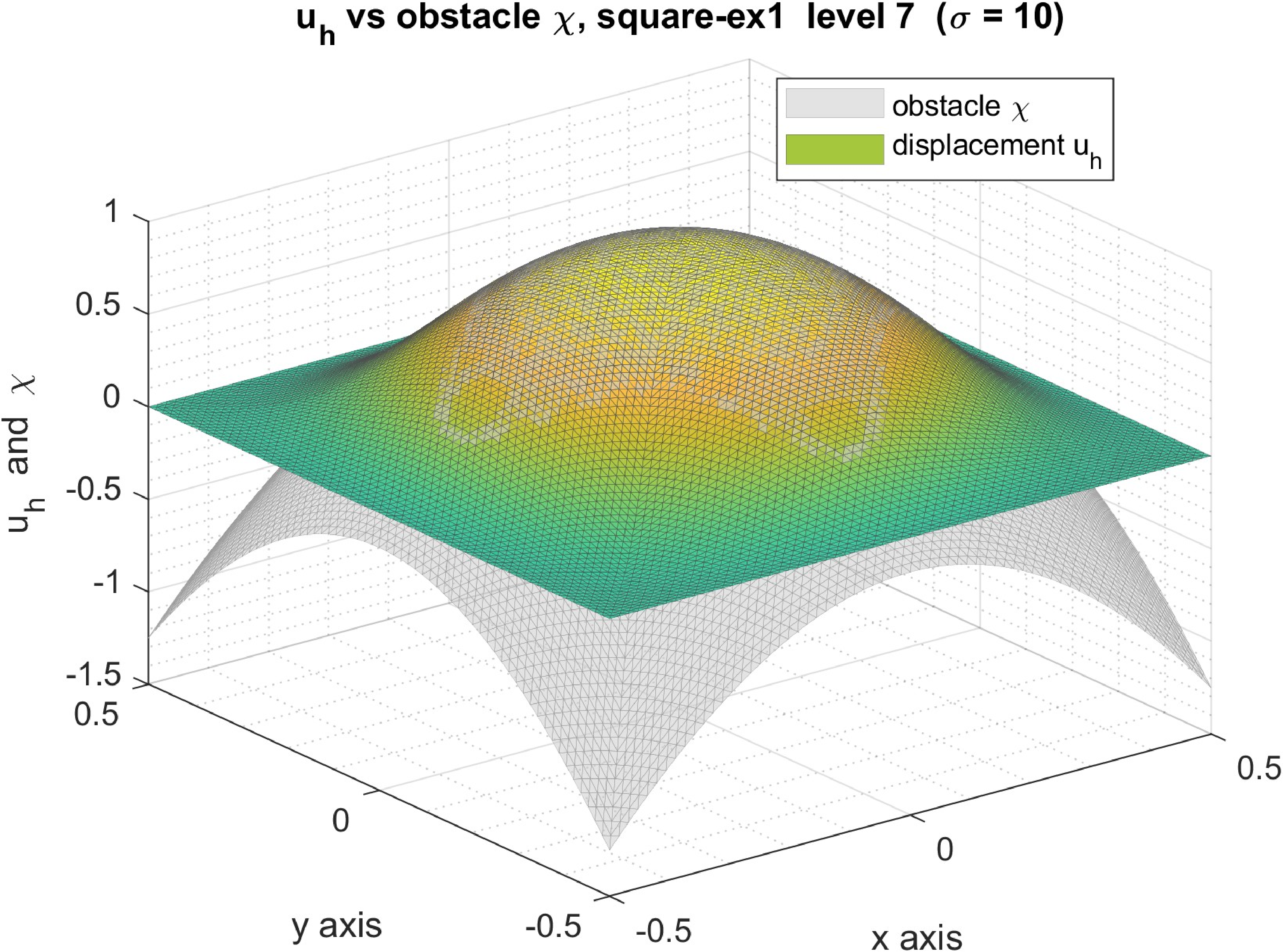}}
\vspace{0.4cm}
\subfloat[$\cC_6$]{%
    \includegraphics[width=0.42\textwidth]{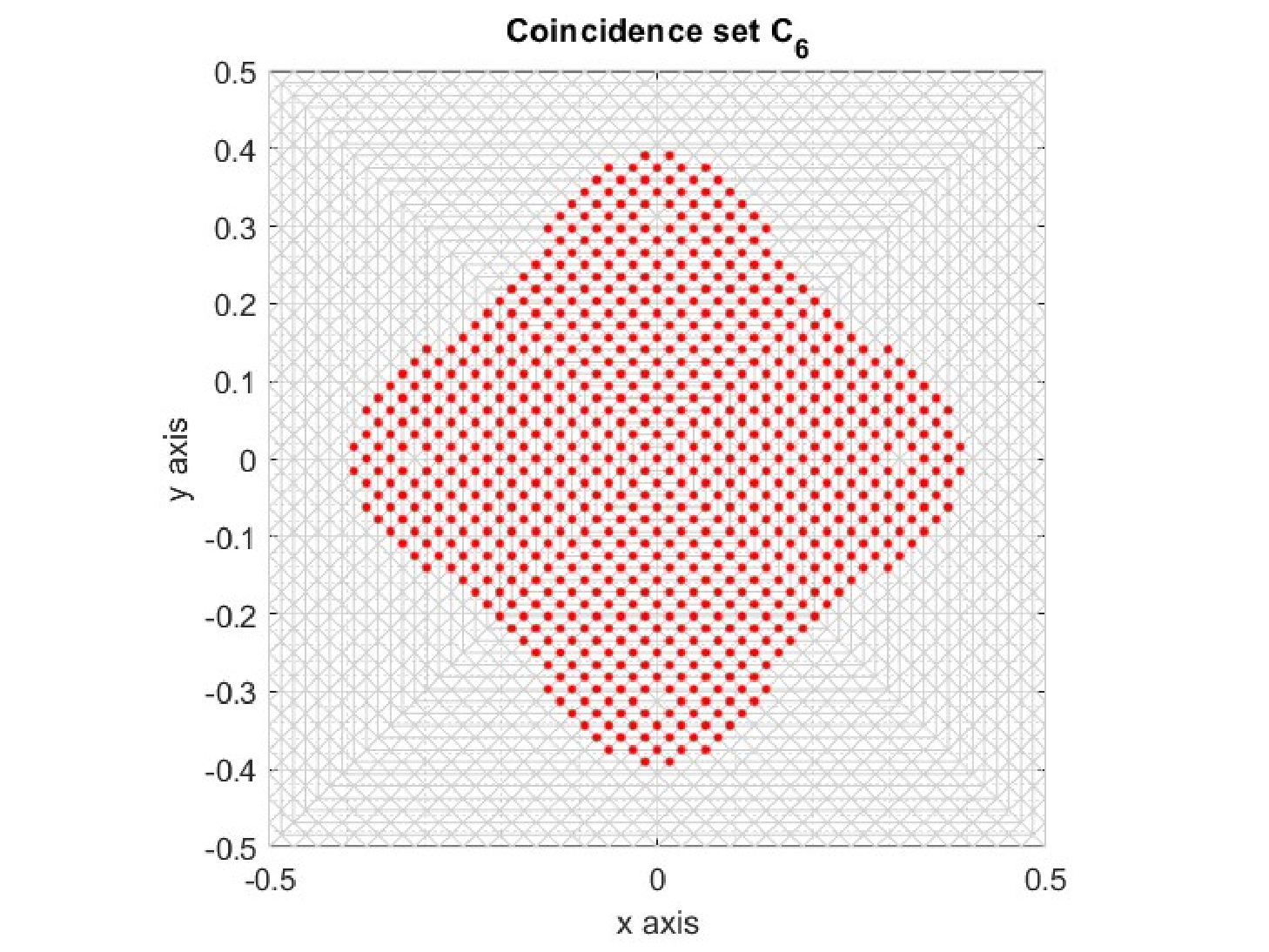}}
\hfill
\subfloat[$\cC_7$]{%
    \includegraphics[width=0.42\textwidth]{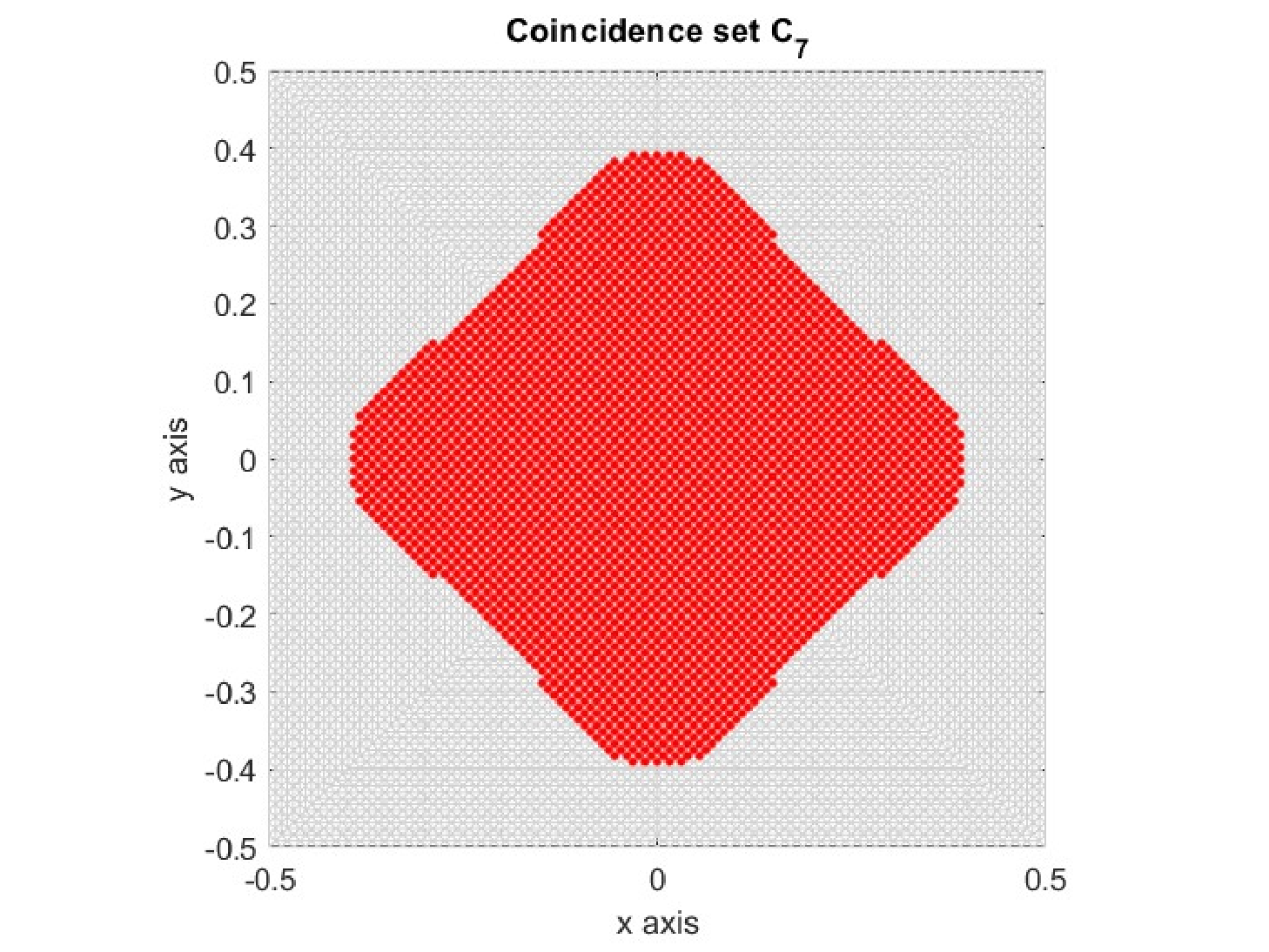}}
\caption{\Cref{ex:sq1} at $\sigma=10$: the discrete displacement against the
obstacle on the finest level, and the discrete coincidence sets on the two
finest levels.}
\label{fig:ex1}

{\footnotesize Alt text: Three panels for the square with obstacle
$\chi=1-5|x|^2+|x|^4$. The upper panel is a surface plot of the discrete
displacement lying above the obstacle surface. The lower two panels are
scatter plots of the coincidence set at levels six and seven; both are filled
diamond-shaped regions centred at the origin, the finer level denser.}
\end{figure}

\Cref{tab:ex1} shows orders between $1.23$ and $1.45$ for the displacement and between $1.09$ and $1.40$ for the Airy stress function in the discrete energy norm, both above the order $\alpha=1$ of \Cref{thm:err}. The orders in the
maximum norm lie between $1.54$ and $2.31$ for $v$ and between $1.08$ and $2.03$ for $u$; \citep[Thm.~4.1]{brenner2012quadratic} proves the order $\alpha$
in $L^\infty(\O)$ for the biharmonic obstacle problem and observes orders near $2$. The value $\widetilde e_1(u)=0$ is exact, because the only interior vertex of $\T_1$ is the centre of the criss-cross mesh and it lies in the coincidence
set on every level.

\begin{example}[contact set of vanishing measure]\label{ex:sq2}
Reversing the sign of the quartic term, $\chi(x)=1-5|x|^2-|x|^4$, gives
$\Delta^2\chi=-64<0$. Interior points are then excluded from the coincidence
set by \citep{caffarelli1979obstacle}, so its interior is empty, and the sets
computed in \Cref{fig:ex2} collapse onto a closed curve.
\end{example}

\begin{table}[htbp]
\centering
\footnotesize
\setlength{\tabcolsep}{3.5pt}
\caption{Convergence history for \Cref{ex:sq2} at $\sigma=10$, on the same sequence of triangulations as \Cref{tab:ex1}.}
\label{tab:ex2}
\begin{tabular}{||c|c||r|c||r|c||r|c||r|c||}
\hline\hline
$\ell$ & $h_\ell$ & $\widetilde e_\ell(u)$ & EOC & $\widetilde e_\ell(v)$ & EOC & $e_\ell(u)$ & EOC & $e_\ell(v)$ & EOC \\
\hline\hline
1 & 1.0000 & 0.003901 & 0.7114 & 0.111758 & 1.7454 & 161.519430 & 1.2821 & 5.263087 & 1.0853 \\
2 & 0.5000 & 0.028739 & 1.6095 & 0.045039 & 1.8540 & 78.731664 & 1.3435 & 3.534316 & 1.2131 \\
3 & 0.2500 & 0.051197 & 2.4236 & 0.008259 & 1.6563 & 34.660280 & 1.3967 & 1.850405 & 1.3062 \\
4 & 0.1250 & 0.005328 & 2.0032 & 0.004711 & 2.0794 & 13.380092 & 1.4085 & 0.790931 & 1.3462 \\
5 & 0.0625 & 0.001212 & 1.8704 & 0.001269 & 2.2664 & 5.095469 & 1.4242 & 0.323883 & 1.4043 \\
6 & 0.0312 & 0.000332 & -- & 0.000264 & -- & 1.898767 & -- & 0.122360 & -- \\
\hline\hline
\end{tabular}
\end{table}

\begin{figure}[htbp]
\centering
\subfloat[$u_h$ and $\chi$, level $7$]{%
    \includegraphics[width=0.60\textwidth]{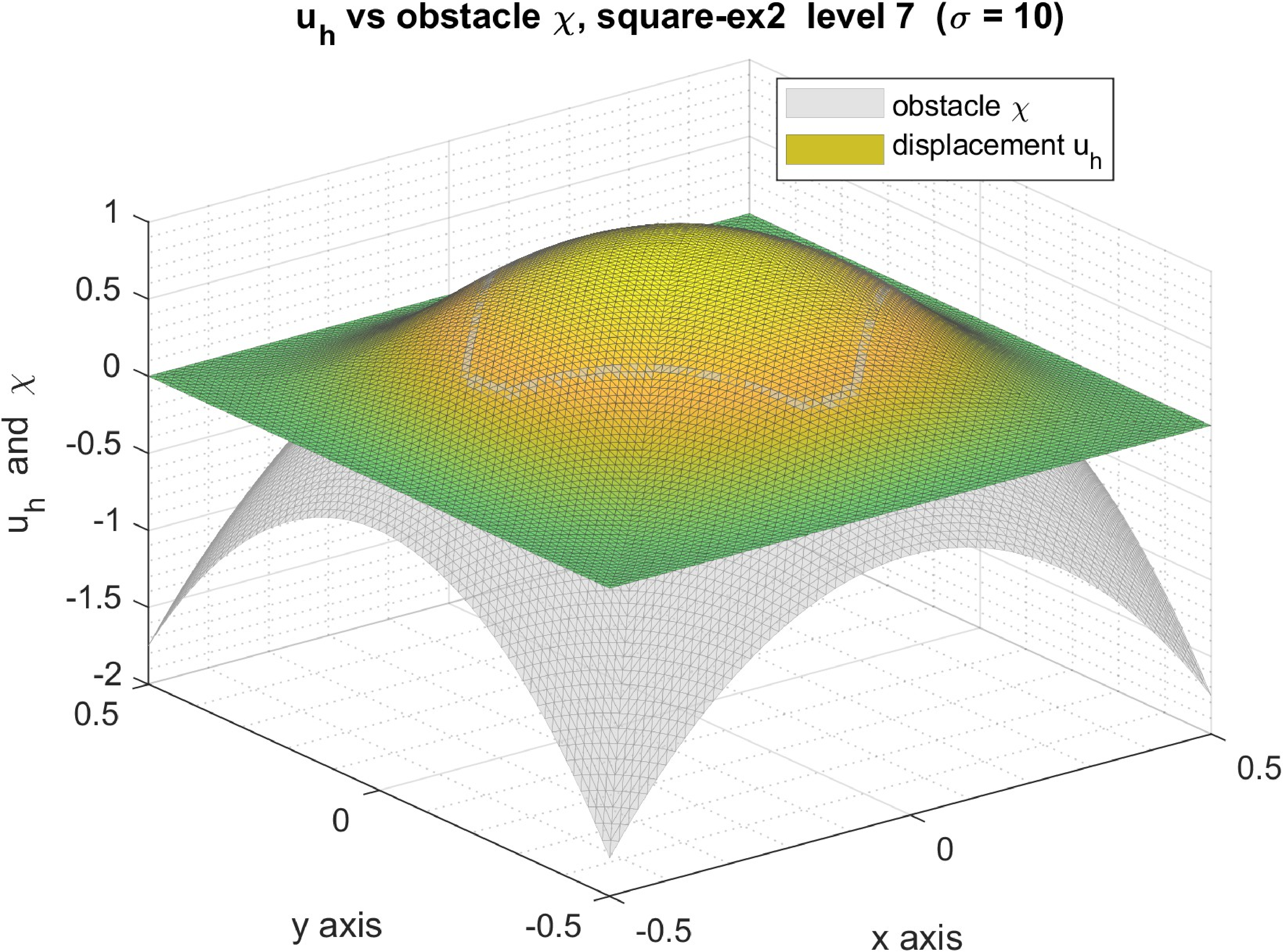}}
\vspace{0.4cm}
\subfloat[$\cC_6$]{%
    \includegraphics[width=0.42\textwidth]{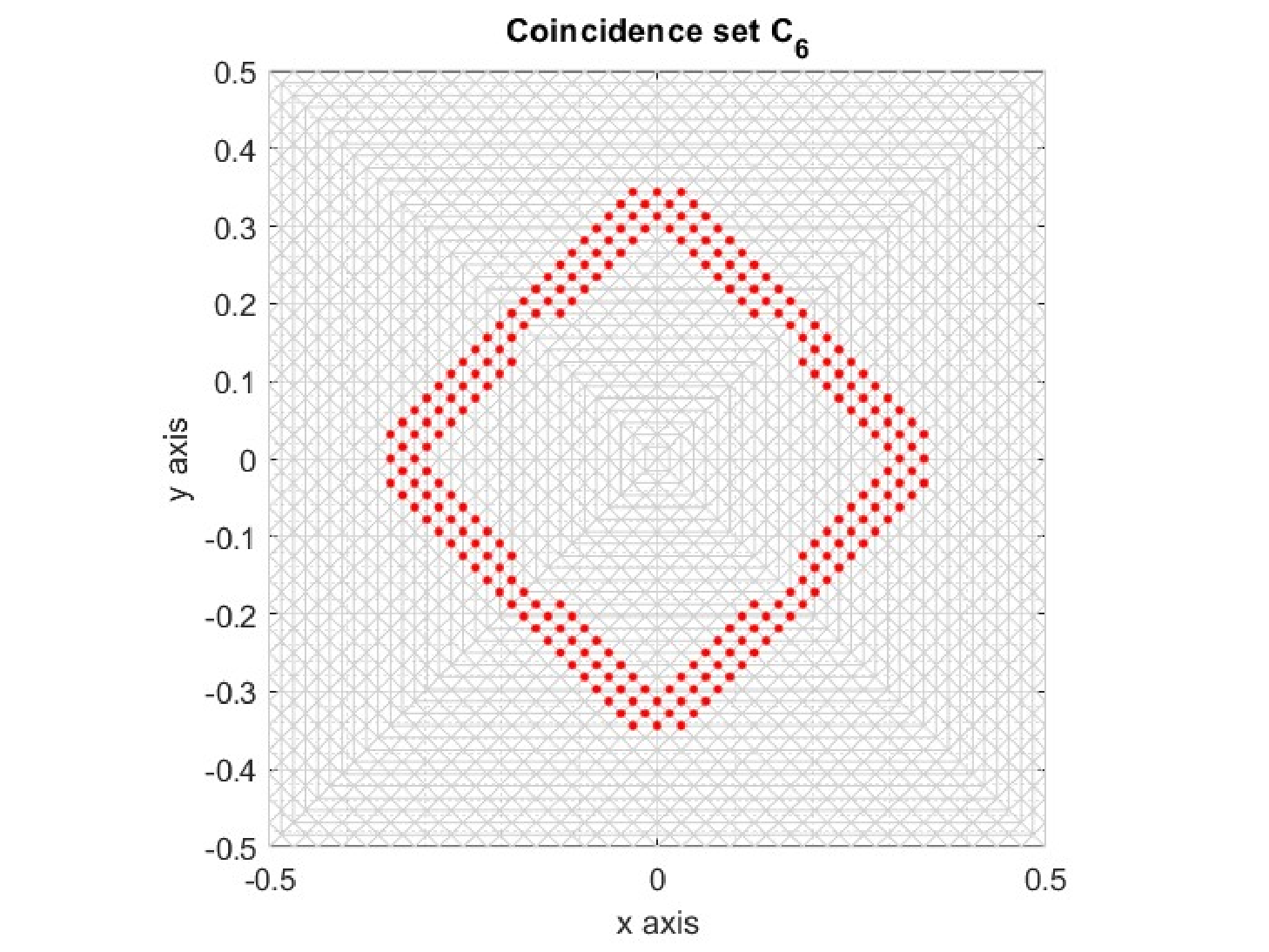}}
\hfill
\subfloat[$\cC_7$]{%
    \includegraphics[width=0.42\textwidth]{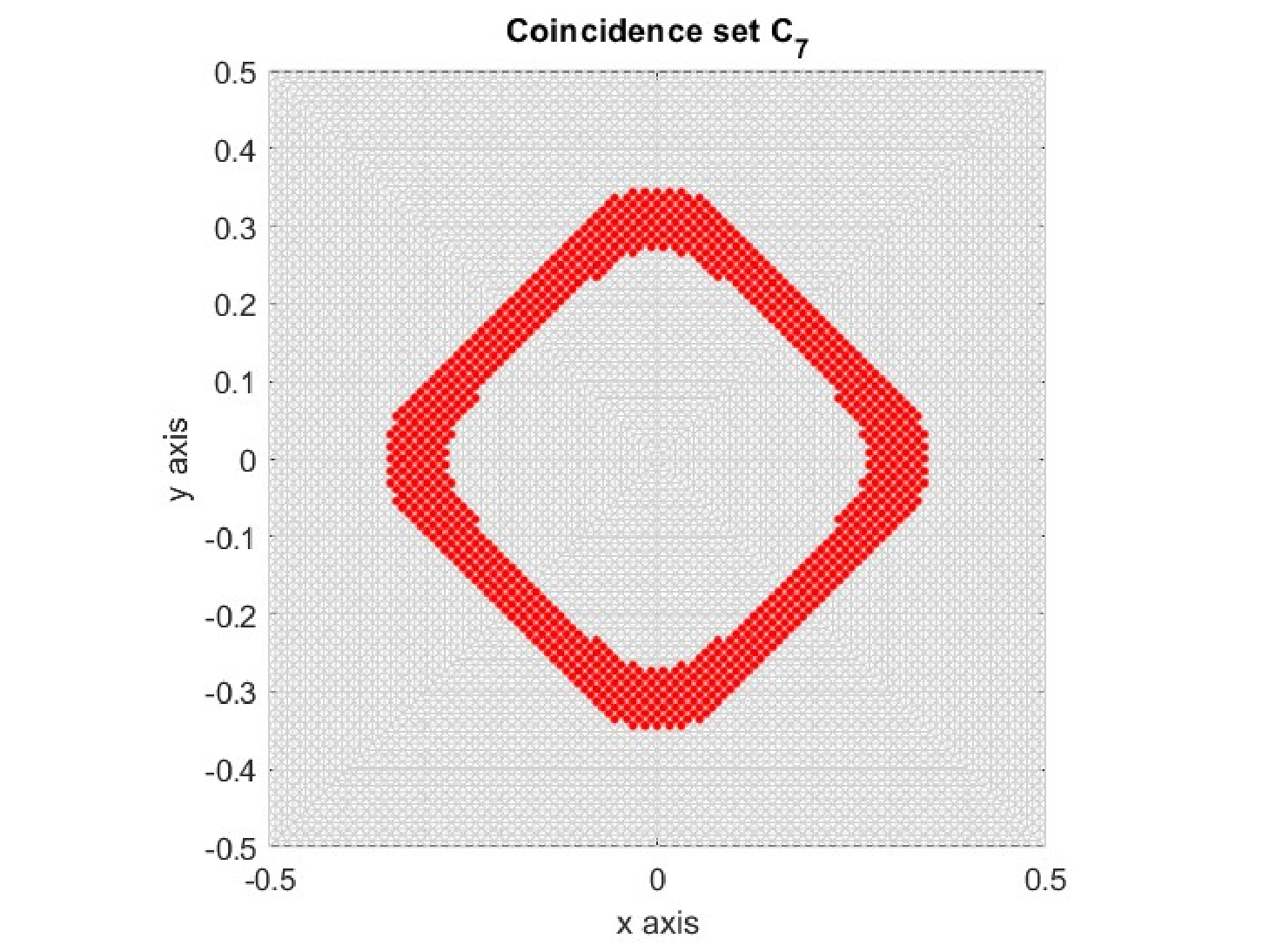}}
\caption{\Cref{ex:sq2} at $\sigma=10$: the discrete displacement against the
obstacle on the finest level, and the discrete coincidence sets on the two
finest levels, which degenerate to a curve.}
\label{fig:ex2}

{\footnotesize Alt text: Three panels for the square with obstacle
$\chi=1-5|x|^2-|x|^4$. The upper panel is a surface plot of the discrete
displacement lying above the obstacle surface. The lower two panels are
scatter plots of the coincidence set at levels six and seven; both form a thin
closed ring rather than a filled region.}
\end{figure}

\Cref{tab:ex2} shows orders between $1.28$ and $1.42$ for the displacement and
between $1.09$ and $1.40$ for the stress function, so the degeneracy of the
coincidence set leaves the rates in the discrete energy norm unchanged. The
nodal errors are less regular: $\widetilde e_\ell(u)$ increases over the three
coarsest levels in both examples, and the least squares order over all levels
is $1.38$ for \Cref{ex:sq1} against $0.99$ for \Cref{ex:sq2}.

\subsection{The obstacle problem on the L-shaped domain}\label{ssec:lshape}

Let $\O=(-\tfrac12,\tfrac12)^2\setminus[0,\tfrac12]^2$ with $f\equiv0$ and
\begin{align}\label{lshapeobstacle}
 \chi(x,y)=1-\frac{(x+0.25)^2}{0.2^2}-\frac{y^2}{0.35^2},
\end{align}
as in \citep{brenner2012quadratic,brenner2013morley,carstensen2021morley}. The
initial triangulation consists of six right-angled triangles with $h_1=0.7071$. The reentrant corner lowers the
elliptic regularity index to $\alpha=0.5445$, so \Cref{thm:err} guarantees only
$\mathcal{O}(h^{0.5445})$.

\begin{table}[htbp]
\centering
\footnotesize
\setlength{\tabcolsep}{3.5pt}
\caption{Convergence history on the L-shaped domain at $\sigma=10$. The six levels carry $5$, $33$, $161$, $705$, $2945$ and $12033$ degrees of freedom.}
\label{tab:lshape}
\begin{tabular}{||c|c||r|c||r|c||r|c||r|c||}
\hline\hline
$\ell$ & $h_\ell$ & $\widetilde e_\ell(u)$ & EOC & $\widetilde e_\ell(v)$ & EOC & $e_\ell(u)$ & EOC & $e_\ell(v)$ & EOC \\
\hline\hline
1 & 0.7071 & 0.000000 & -- & 0.000000 & -- & 60.498653 & 0.7057 & 3.328834 & 0.7147 \\
2 & 0.3536 & 0.151159 & 1.7950 & 0.068928 & 2.0744 & 124.686538 & 1.2887 & 3.931121 & 1.0329 \\
3 & 0.1768 & 0.038592 & 1.7077 & 0.016832 & 2.0947 & 56.822099 & 1.3662 & 2.311176 & 1.1661 \\
4 & 0.0884 & 0.012566 & 1.7966 & 0.004924 & 2.4161 & 22.030216 & 1.3654 & 1.150271 & 1.3256 \\
5 & 0.0442 & 0.003617 & -- & 0.000923 & -- & 8.550579 & -- & 0.458944 & -- \\
\hline\hline
\end{tabular}
\end{table}

\begin{figure}[htbp]
\centering
\subfloat[$u_h$ and $\chi$, level $6$]{%
    \includegraphics[width=0.60\textwidth]{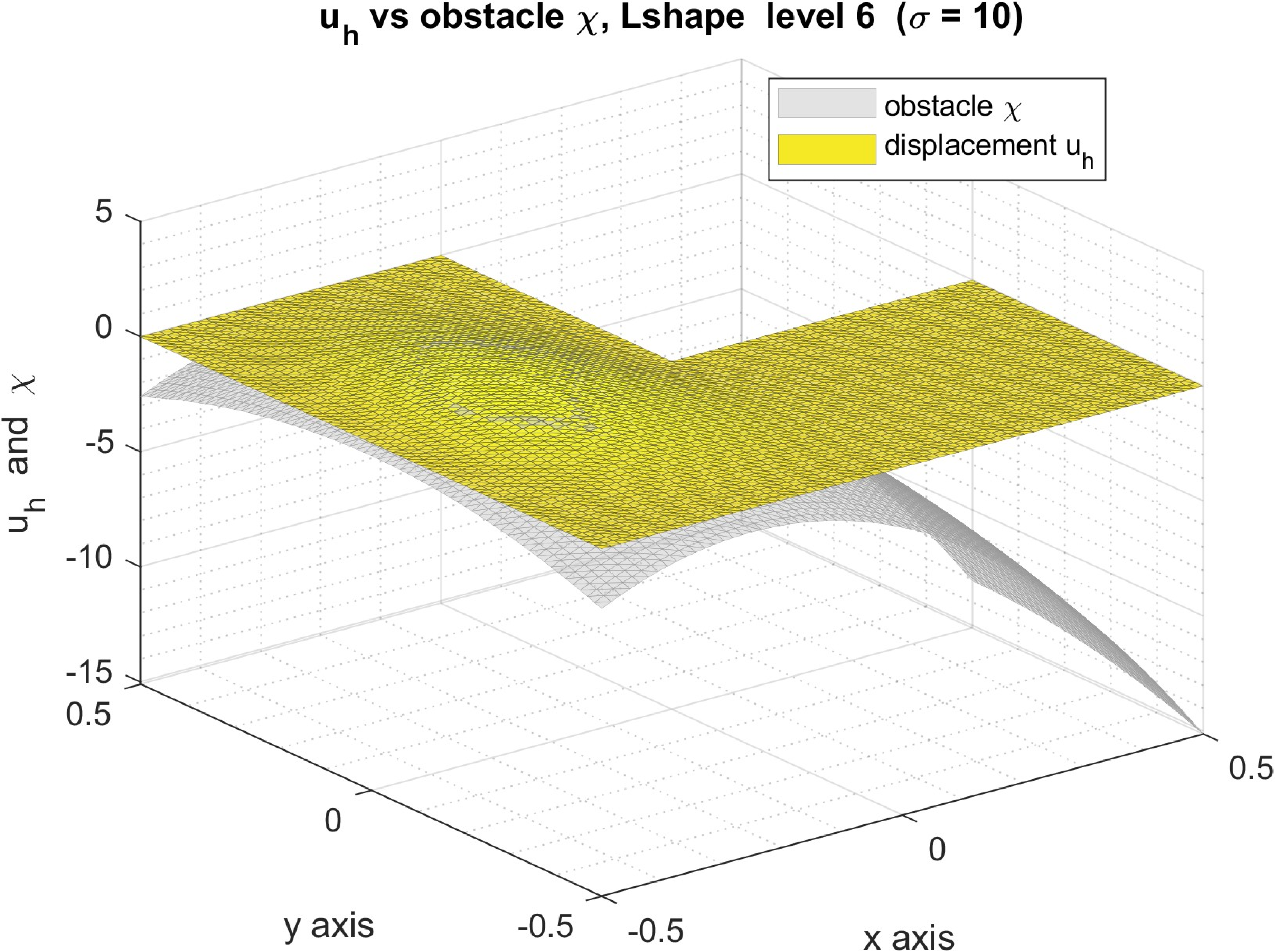}}
\vspace{0.4cm}
\subfloat[$\cC_5$]{%
    \includegraphics[width=0.42\textwidth]{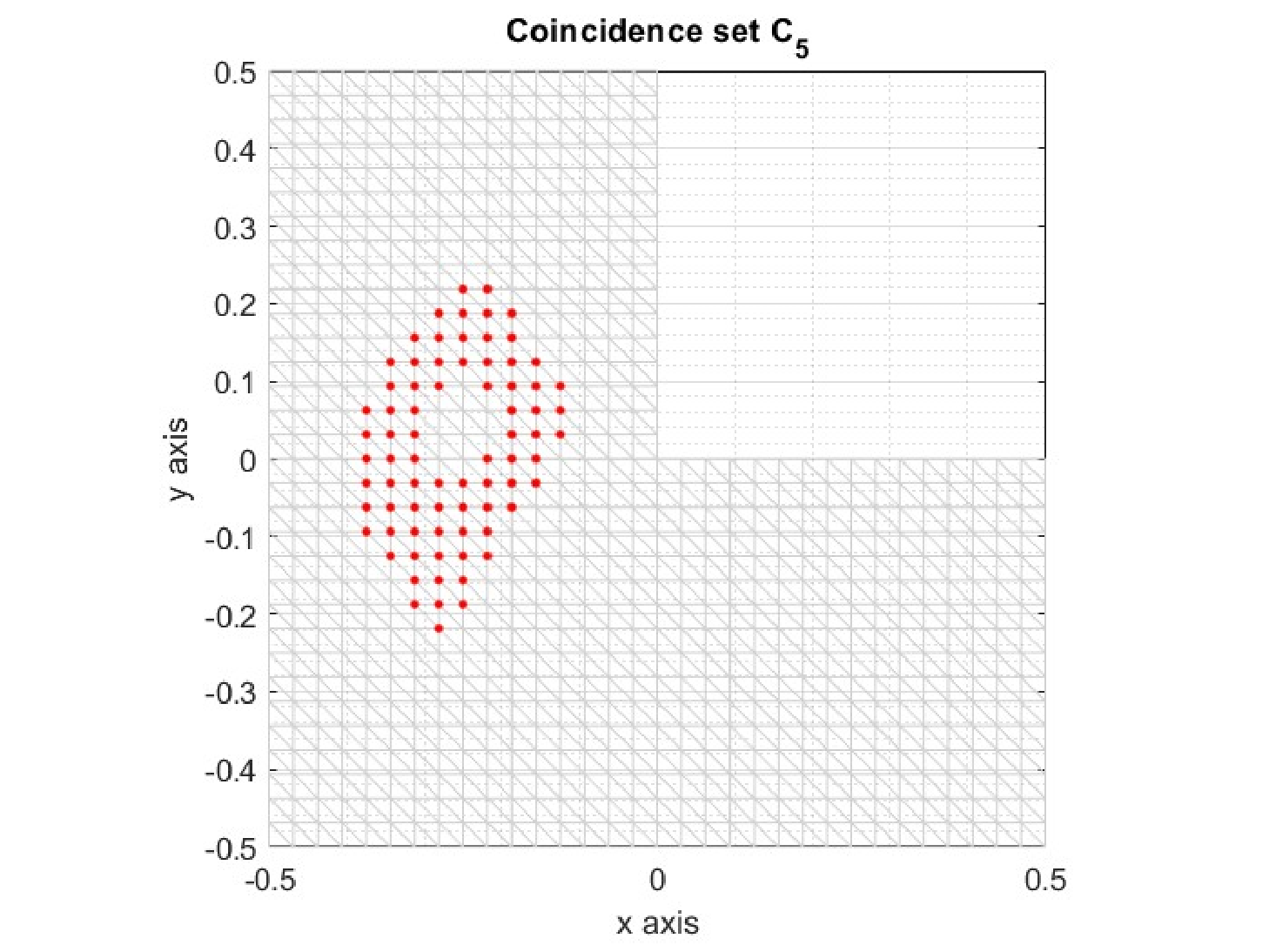}}
\hfill
\subfloat[$\cC_6$]{%
    \includegraphics[width=0.42\textwidth]{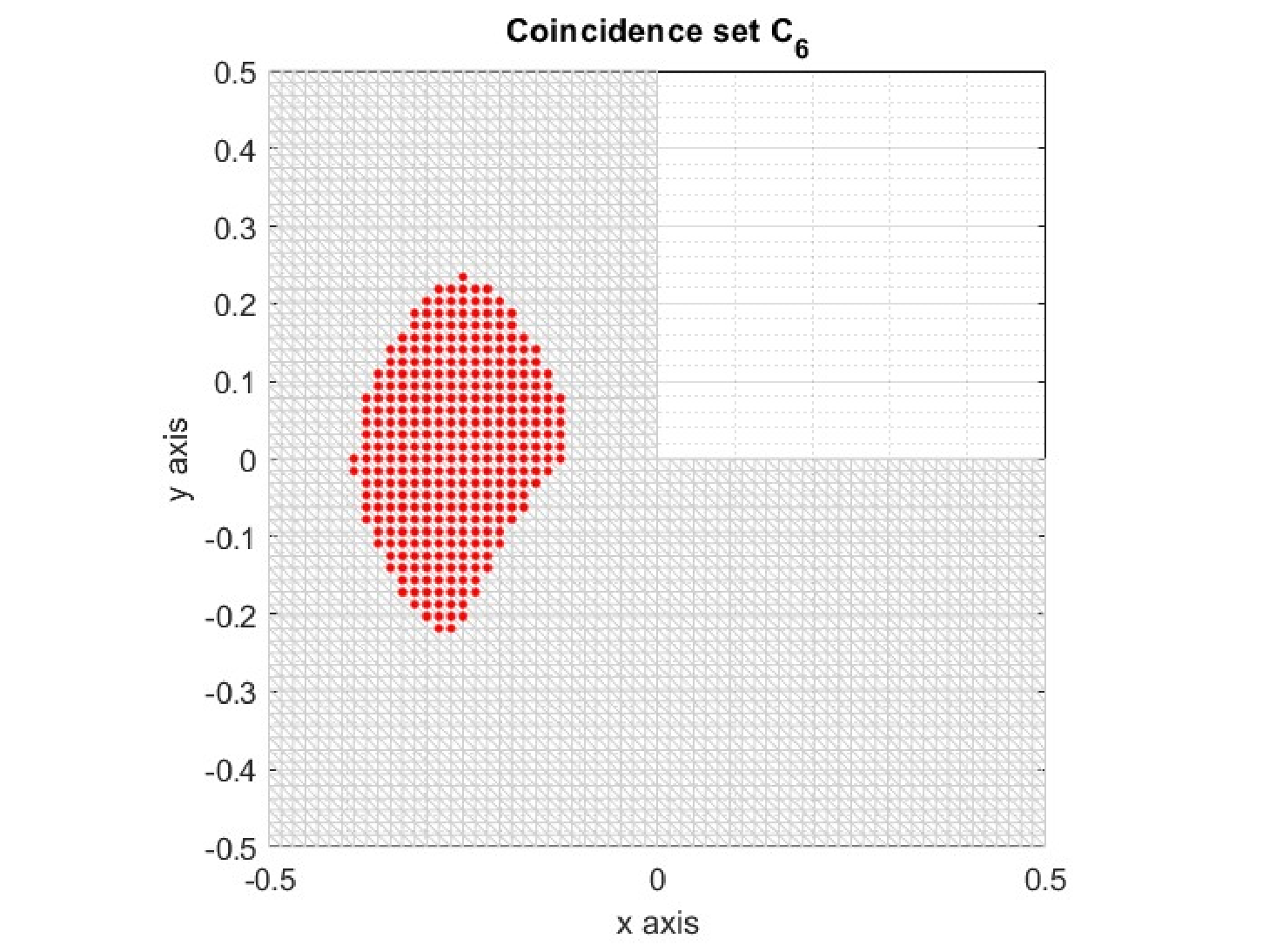}}
\caption{L-shaped domain at $\sigma=10$: the discrete displacement against the
obstacle on the finest level, and the discrete coincidence sets on the two
finest levels.}
\label{fig:lshape}

{\footnotesize Alt text: Three panels for the L-shaped domain. The upper panel
is a surface plot of the discrete displacement lying above the obstacle
surface. The lower two panels show the coincidence set at levels five and six
as a single compact blob in the left part of the domain, away from the
reentrant corner.}
\end{figure}

All eight vertices of $\T_1$ lie on $\partial\O$, so $V_h$ is spanned by five
edge-midpoint functions and $\widetilde e_1(u)=\widetilde e_1(v)=0$. From the
second level on, \Cref{tab:lshape} shows an order close to $1.37$ for the
displacement and an order rising from $1.03$ to $1.33$ for the stress function.
Both exceed the guaranteed exponent $0.5445$, so the computations are in the
pre-asymptotic range; the same is observed for the Morley discretisation of
this example in \citep{carstensen2021morley} and for the biharmonic obstacle
problem in \citep[Ex.~4]{brenner2012quadratic}. Since $\Delta^2\chi=0$ for the
obstacle \eqref{lshapeobstacle}, the complement of the coincidence set is
connected, as \Cref{fig:lshape} confirms.

\subsection{Dependence on the penalty parameter}\label{ssec:sigma}

The three experiments were repeated for nine penalties,
$\sigma\in\{5,10,15,20,25,30,35,40,45\}$. \Cref{tab:sigma} lists the energy
norm order obtained from a least squares fit of $\log_2 e_\ell$ against $\ell$
over all levels with a nonvanishing difference, and \Cref{fig:sigma} displays
the same information together with its maximum norm counterpart.

\begin{table}[htbp]
\centering
\footnotesize
\caption{Least squares discrete energy norm orders for nine values of the penalty
parameter. The guaranteed exponent is $\alpha=1$ on the square and
$\alpha=0.5445$ on the L-shaped domain.}
\label{tab:sigma}
\begin{tabular}{||c||c|c||c|c||c|c||}
\hline\hline
& \multicolumn{2}{c||}{\Cref{ex:sq1}} & \multicolumn{2}{c||}{\Cref{ex:sq2}}
& \multicolumn{2}{c||}{L-shape}\\
\cline{2-7}
$\sigma$ & $u$ & $v$ & $u$ & $v$ & $u$ & $v$\\
\hline\hline
5 & 1.0909 & 1.2023 & 1.1713 & 1.1915 & 0.7836 & 0.7988 \\
10 & 1.2332 & 1.1162 & 1.2936 & 1.1058 & 0.8146 & 0.7490 \\
15 & 1.2779 & 1.0613 & 1.3326 & 1.0522 & 0.8136 & 0.7083 \\
20 & 1.3048 & 1.0184 & 1.3563 & 1.0103 & 0.8083 & 0.6774 \\
25 & 1.3235 & 0.9831 & 1.3728 & 0.9755 & 0.8021 & 0.6524 \\
30 & 1.3376 & 0.9529 & 1.3851 & 0.9460 & 0.7956 & 0.6312 \\
35 & 1.3486 & 0.9266 & 1.3947 & 0.9202 & 0.7893 & 0.6128 \\
40 & 1.3575 & 0.9034 & 1.4024 & 0.8976 & 0.7831 & 0.5964 \\
45 & 1.3649 & 0.8826 & 1.4088 & 0.8771 & 0.7770 & 0.5818 \\
\hline\hline
\end{tabular}
\end{table}

\begin{figure}[htbp]
\centering
\subfloat[energy norm]{\includegraphics[width=0.46\textwidth]{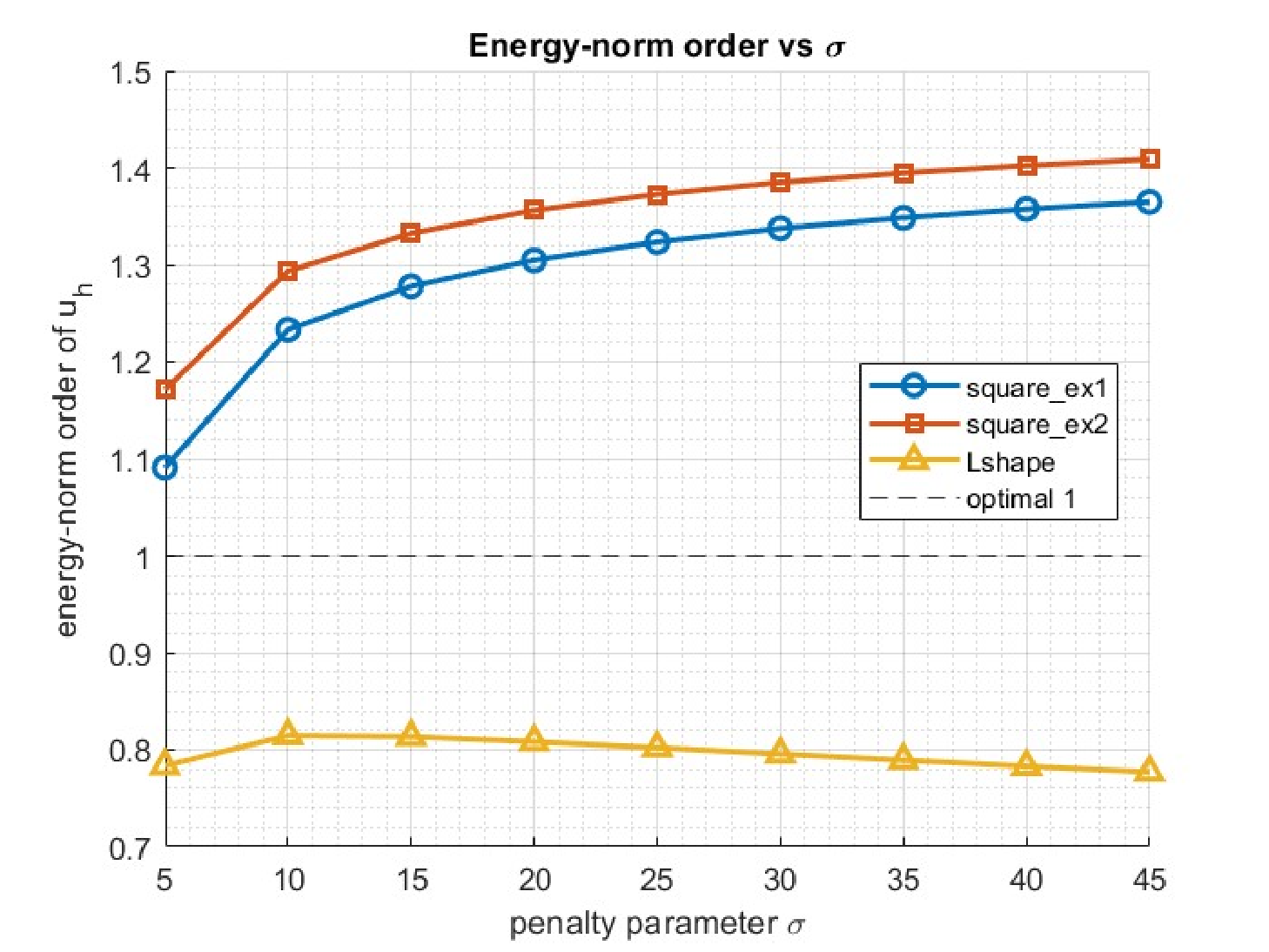}}
\hfill
\subfloat[maximum norm]{\includegraphics[width=0.46\textwidth]{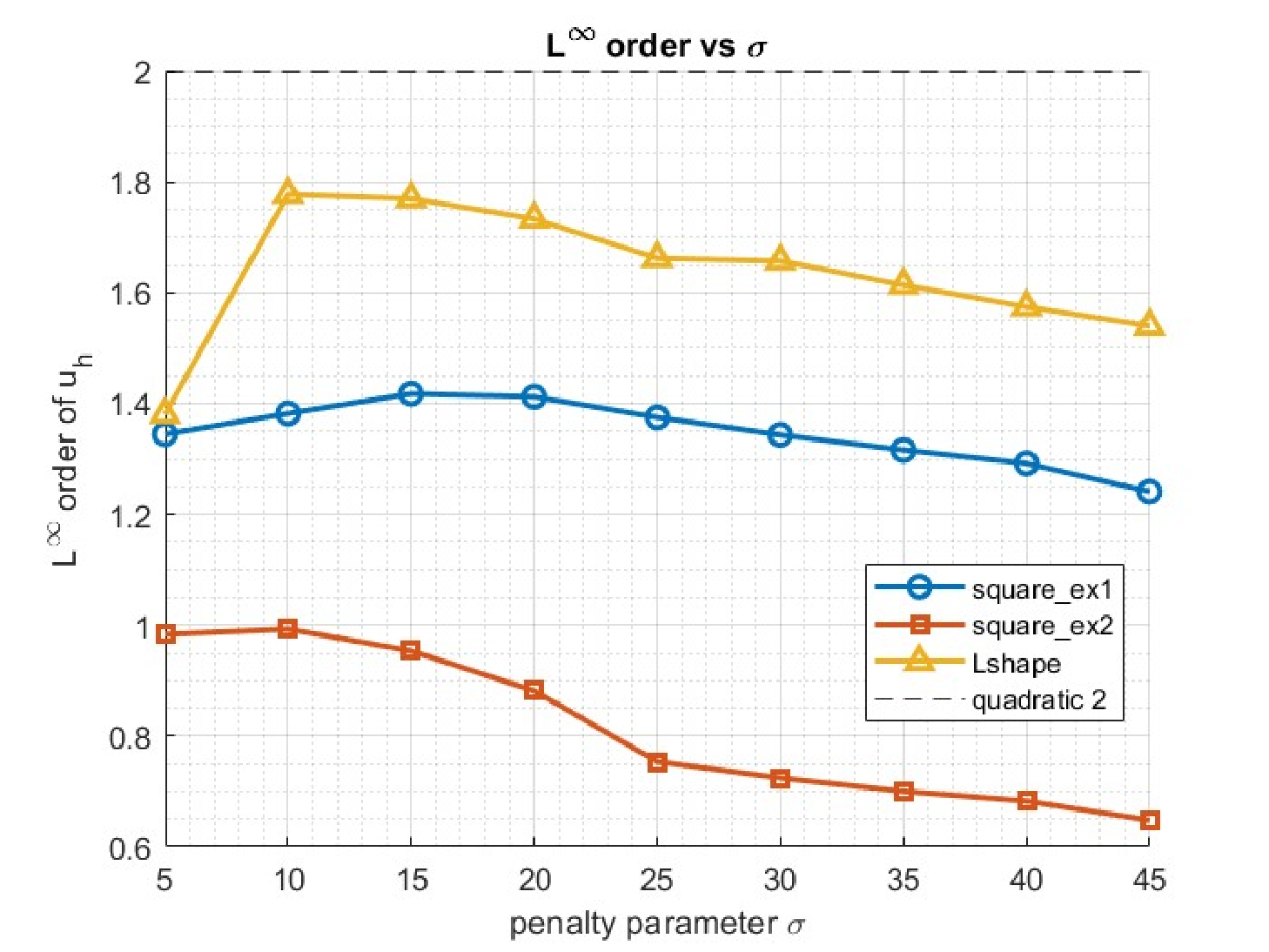}}
\caption{Observed orders of convergence as functions of the penalty parameter
$\sigma$ for the three experiments.}
\label{fig:sigma}

{\footnotesize Alt text: Two line plots of observed convergence order against
the penalty parameter $\sigma_1$ from 5 to 45, one curve per experiment. In
the energy norm the displacement curves rise with $\sigma_1$ and the L-shape
curve is nearly flat; in the maximum norm the curves decrease after an initial
rise.}
\end{figure}

The dependence is monotone. On the square the order for the displacement grows
with the penalty, from $1.09$ to $1.36$ in \Cref{ex:sq1} and from $1.17$ to
$1.41$ in \Cref{ex:sq2}, while the order for the Airy stress function falls,
from $1.20$ to $0.88$ and from $1.19$ to $0.88$; the two cross between
$\sigma=5$ and $\sigma=10$. For $\sigma\le20$ both components exceed the
order $\alpha=1$ of \Cref{thm:err} on the square, and from $\sigma=25$
onwards the stress function falls below it. On the L-shaped domain all eighteen
entries exceed the guaranteed exponent $0.5445$, the displacement order varying
only between $0.777$ and $0.815$.

Among the nine values, $\sigma=10$ gives the best worst case: it maximises the
smallest ratio of observed order to guaranteed order over the six columns of
\Cref{tab:sigma}. It is on that basis, and not because it produces the largest
individual rate, that the detailed tables of \Cref{ssec:square} and
\Cref{ssec:lshape} use $\sigma=10$. The whole range lies above the value $\sigma=5$ for which $a_{\rm IP}(\bullet,\bullet)$ is known to be coercive \citep[Rem.~2.3]{brenner2012quadratic}, and no instability was met anywhere in it, the computations completing on every level for all nine values.

\subsection{Violation of the smallness assumption}\label{ssec:smallness}

The uniqueness criteria of \Cref{thm:contsdep} and \Cref{thm:disc-existence}
demand small data. This subsection quantifies the violation for the load
$f(x,y)=(x+3)^2(x-3)^2(y+3)^2(y-3)^2$ on the square of \Cref{ex:sq1} and
records how the iteration of \Cref{ssec:solver} behaves once the hypothesis
fails. All runs use $\sigma=10$. The outer loop is allowed $100$ iterations here rather than the forty of \Cref{ssec:solver}, matching \citep{carstensen2021morley}.

Any single admissible function bounds the embedding constants of
\eqref{ctsembedding} from below. The choice
$w(x,y)=\bigl(\tfrac14-x^2\bigr)^2\bigl(\tfrac14-y^2\bigr)^2\in H^2_0(\O)$,
whose maximum is attained at the origin, gives
$\|w\|_{L^2(\O)}=0.0015873$, $\|w\|_{L^\infty(\O)}=0.0039063$ and
$\trinl w\trinr=0.0571429$, whence
\begin{align}\label{constbounds}
 C_{\rm F}\ge\frac{\|w\|_{L^2(\O)}}{\trinl w\trinr}=0.0278,
 \qquad
 C_{\rm S}\ge\frac{\|w\|_{L^\infty(\O)}}{\trinl w\trinr}=0.0684 .
\end{align}
The definition of $M(f,\chi)$ in \Cref{thm:contsdep} gives
$M(f,\chi)\ge\sqrt3\,C_{\rm F}\|f\|_{L^2(\O)}$, and
$\|f\|_{L^2(\O)}=6323.9951$, so that
\begin{align}\label{smallviolated}
 C_{\rm S}M(f,\chi)\ \ge\ \sqrt3\,C_{\rm S}C_{\rm F}\|f\|_{L^2(\O)}
 \ =\ 20.7993\ \gg\ \tfrac12 .
\end{align}
The sufficient condition of \Cref{thm:err} is therefore violated by a factor of
more than forty. All quantities in \eqref{constbounds} and
\eqref{smallviolated} belong to the continuous problem and are independent of
$\sigma$ and of the triangulation; they agree with the bounds obtained for
the Morley discretisation of the same example in \citep{carstensen2021morley}.

For this load the outer loop nevertheless terminated on all six levels, whereas
the Morley discretisation of the same example does not converge within $100$
iterations \citep{carstensen2021morley}. Failure of the smallness condition thus
need not obstruct termination, but it withdraws the guarantee that the computed
pair is the only discrete solution.

Scaling the obstacle of \Cref{ex:sq1} to $\lambda\chi$ with
$\lambda\in\{4,\dots,10\}$ is more revealing, since it enlarges the coincidence
set while keeping $f\equiv0$. \Cref{tab:lambda} records, level by level,
whether the outer loop terminated within the permitted number of iterations.

\begin{table}[htbp]
\centering
\footnotesize
\caption{Termination of the outer loop for the scaled obstacle $\lambda\chi$ at
$\sigma=10$. A tick marks the levels on which the partition into active and
inactive vertices stagnated.}
\label{tab:lambda}
\begin{tabular}{||c||c|c|c|c|c|c||c||}
\hline\hline
$\lambda$ & $\T_1$ & $\T_2$ & $\T_3$ & $\T_4$ & $\T_5$ & $\T_6$ & terminated\\
\hline\hline
4 & \checkmark & \checkmark & \checkmark & \checkmark & \checkmark & \checkmark & 6 \\
5 & \checkmark & \checkmark & \checkmark & \checkmark & \checkmark & \checkmark & 6 \\
6 & \checkmark & \checkmark & \checkmark & \checkmark & \checkmark & \checkmark & 6 \\
7 & \checkmark & \checkmark & \checkmark & \checkmark & \checkmark & \checkmark & 6 \\
8 & \checkmark & \checkmark & \checkmark & \checkmark & $\times$ & $\times$ & 4 \\
9 & \checkmark & \checkmark & $\times$ & $\times$ & $\times$ & $\times$ & 2 \\
10 & \checkmark & \checkmark & $\times$ & $\times$ & $\times$ & $\times$ & 2 \\
\hline\hline
\end{tabular}
\end{table}

\Cref{tab:lambda} exhibits a threshold. For $\lambda\le7$ the outer loop
terminates on all six levels. At $\lambda=8$ it terminates on the four coarsest
triangulations and fails on the two finest, and at $\lambda=9$ and
$\lambda=10$ it fails from the third level onwards. The number of levels on
which the partition stagnates therefore decreases monotonically from six to
four to two as the obstacle grows, and the failures appear first on the finest
meshes. This is the regime in which \Cref{rem:disc-smallness} no longer
guarantees a unique discrete solution, and the same behaviour is reported for
the Morley discretisation of this example in \citep{carstensen2021morley}.
A convergence theory for the coupled active set and Newton iteration in the
presence of the semilinearity remains open.

\section{Conclusions}\label{sec:conclusion}

This article analyses a quadratic $C^0$ interior penalty method for the von
K\'arm\'an obstacle problem. The discrete space consists of Lagrange $P_2$
elements and the obstacle constraint acts at the vertices. The trilinear form
contains terms on the edges, is symmetric in its first two arguments and is
bounded in the discrete energy norm. \Cref{thm:disc-sob} gives the Sobolev and
Friedrichs constants of the discrete energy norm with an explicit power of the mesh
size, and these constants enter the smallness condition of
\Cref{rem:disc-smallness}. The discrete problem has a solution, this solution
is unique under that condition, and the energy norm error is of order
$\mathcal{O}(h^\alpha)$ by \Cref{thm:err}.

The numerical experiments of \Cref{sec:numerics} confirm these rates on a
square domain and show the two possible forms of the coincidence set. On the
L-shaped domain the observed rates exceed the order that the reentrant corner
imposes, so the computations are in the pre-asymptotic range. The observed
rates depend on the penalty parameter in a monotone way, and the experiments
identify the range $\sigma\le20$ in which both solution components attain
the predicted order. For large data the smallness condition fails, and the
iterative solver then loses its reliability on fine meshes. An a posteriori
error analysis with adaptive mesh refinement for \eqref{C0IPwform} and a
convergence proof for the iterative solver remain open.

\bibliographystyle{plainnat}
\bibliography{vKeBib}

@article{carstensen2021morley,
  title={Morley finite element method for the von K{\'a}rm{\'a}n obstacle problem},
  author={Carstensen, Carsten and Gaddam, Sharat and Nataraj, Neela and Pani, Amiya K and Shylaja, Devika},
  journal={ESAIM: Mathematical Modelling and Numerical Analysis},
  volume={55},
  number={5},
  pages={1873--1894},
  year={2021},
  publisher={EDP Sciences}
}

@article{miersemann1992stability,
  title={Stability in obstacle problems for the von Karman plate},
  author={Miersemann, E and Mittelmann, HD},
  journal={SIAM journal on mathematical analysis},
  volume={23},
  number={5},
  pages={1099--1116},
  year={1992},
  publisher={SIAM}
}

@incollection{bacuta2002shift,
  title={Shift theorems for the biharmonic Dirichlet problem},
  author={Bacuta, Constantin and Bramble, James H and Pasciak, Joseph E},
  booktitle={Recent Progress in Computational and Applied PDEs},
  pages={1--26},
  year={2002},
  publisher={Springer}
}

@article{berger1968karman,
	title={Von {K{\'a}rm{\'a}n} equations and the buckling of a thin elastic plate, II plate with general edge conditions},
	author={Berger, Melvyn S and Fife, Paul C},
	journal={Communications on Pure and Applied Mathematics},
	volume={21},
	number={3},
	pages={227--241},
	year={1968},
	publisher={Wiley Online Library}
}

@article{blum1980boundary,
	title={On the boundary value problem of the biharmonic operator on domains with angular corners},
	author={Blum, Heribert and Rannacher, Rolf and Leis, R},
	journal={Mathematical Methods in the Applied Sciences},
	volume={2},
	number={4},
	pages={556--581},
	year={1980},
	publisher={Wiley Online Library}
}

@article{brenner2017c,
	title={A {$C^0$} interior penalty method for a von {K{\'a}rm{\'a}n} plate},
	author={Brenner, Susanne C and Neilan, Michael and Reiser, Armin and Sung, Li-Yeng},
	journal={Numerische Mathematik},
	volume={135},
	number={3},
	pages={803--832},
	year={2017},
	publisher={Springer}
}

@article {hintermuller2002primaldual,
    AUTHOR = {Hinterm\"{u}ller, M. and Ito, K. and Kunisch, K.},
     TITLE = {The primal-dual active set strategy as a semismooth {N}ewton
              method},
   JOURNAL = {SIAM J. Optim.},
  FJOURNAL = {SIAM Journal on Optimization},
    VOLUME = {13},
      YEAR = {2002},
    NUMBER = {3},
     PAGES = {865--888 (2003)},
      ISSN = {1052-6234},
   MRCLASS = {90C33 (65K10 90C53)},
  MRNUMBER = {1972219},
MRREVIEWER = {Hou Duo Qi},
       DOI = {10.1137/S1052623401383558},
       URL = {https://doi.org/10.1137/S1052623401383558},
}

@article{brenner2012quadratic,
	title={A quadratic {$C^0$} interior penalty method for the displacement obstacle problem of clamped Kirchhoff plates},
	author={Brenner, Susanne C and Sung, Li-Yeng and Zhang, Hongchao and Zhang, Yi},
	journal={SIAM Journal on Numerical Analysis},
	volume={50},
	number={6},
	pages={3329--3350},
	year={2012},
	publisher={SIAM}
}

@article{brenner2013morley,
	title={A Morley finite element method for the displacement obstacle problem of clamped Kirchhoff plates},
	author={Brenner, Susanne C and Sung, Li-yeng and Zhang, Hongchao and Zhang, Yi},
	journal={Journal of Computational and Applied Mathematics},
	volume={254},
	pages={31--42},
	year={2013},
	publisher={Elsevier}
}

@article{brenner2012finite,
	title={Finite element methods for the displacement obstacle problem of clamped plates},
	author={Brenner, Susanne C and Sung, Li-yeng and Zhang, Yi},
	journal={Mathematics of Computation},
	volume={81},
	number={279},
	pages={1247--1262},
	year={2012}
}

@article{brezzi1978finite,
	title={Finite element approximations of the von {K{\'a}rm{\'a}n} equations},
	author={Brezzi, Franco},
	journal={RAIRO. Analyse Num{\'e}rique},
	volume={12},
	number={4},
	pages={303--312},
	year={1978},
	publisher={EDP Sciences}
}

@article{caffarelli1979obstacle,
	title={The obstacle problem for the biharmonic operator},
	author={Caffarelli, Luis A and Friedman, Avner},
	journal={Annali della Scuola Normale Superiore di Pisa-Classe di Scienze},
	volume={6},
	number={1},
	pages={151--184},
	year={1979}
}

@article{carstensen2017nonconforming,
	title={Nonconforming finite element discretization for semilinear problems with trilinear nonlinearity},
	author={Carstensen, Carsten and Mallik, Gouranga and Nataraj, Neela},
	journal={arXiv preprint arXiv:1708.07627},
	year={2017}
}

@article{carstensen2018priori,
	title={A priori and a posteriori error control of discontinuous Galerkin finite element methods for the von {K{\'a}rm{\'a}n} equations},
	author={Carstensen, Carsten and Mallik, Gouranga and Nataraj, Neela},
	journal={IMA Journal of Numerical Analysis},
	volume={39},
	number={1},
	pages={167--200},
	year={2018},
	publisher={Oxford University Press}
}

@article{carstensen2019adaptive,
	title={Adaptive Morley FEM for the von {K{\'a}rm{\'a}n} equations with optimal convergence rates},
	author={Carstensen, Carsten and Nataraj, Neela},
	journal={arXiv preprint arXiv:1908.08013},
	year={2019}
}

@book{ciarlet1997mathematical,
	title={Mathematical Elasticity: Volume II: Theory of Plates},
	author={Ciarlet, Philippe G},
	volume={27},
	year={1997},
	publisher={Elsevier}
}

@inproceedings{frehse1971differenzierbarkeitsproblem,
	title={Zum differenzierbarkeitsproblem bei variationsungleichungen h{\"o}herer ordnung},
	author={Frehse, Jens},
	booktitle={Abhandlungen aus dem Mathematischen Seminar der Universit{\"a}t Hamburg},
	volume={36},
	number={1},
	pages={140--149},
	year={1971},
	organization={Springer}
}

@article{frehse1973regularity,
	title={On the regularity of the solution of the biharmonic variational inequality},
	author={Frehse, Jens},
	journal={Manuscripta Mathematica},
	volume={9},
	number={1},
	pages={91--103},
	year={1973},
	publisher={Springer}
}

@book{glowinski2008lectures,
	title={Lectures on numerical methods for non-linear variational problems},
	author={Glowinski, Roland},
	year={2008},
	publisher={Springer Science \& Business Media}
}

@article{hoppe1994adaptive,
	title={Adaptive multilevel methods for obstacle problems},
	author={Hoppe, Ronald HW and Kornhuber, Ralf},
	journal={SIAM journal on numerical analysis},
	volume={31},
	number={2},
	pages={301--323},
	year={1994},
	publisher={SIAM}
}

@book{kinderlehrer1980introduction,
	title={An introduction to variational inequalities and their applications},
	author={Kinderlehrer, David and Stampacchia, Guido},
	volume={31},
	year={1980},
	publisher={SIAM}
}

@article{knightly1967existence,
	title={An existence theorem for the von {K{\'a}rm{\'a}n} equations},
	author={Knightly, George H},
	journal={Archive for Rational Mechanics and Analysis},
	volume={27},
	number={3},
	pages={233--242},
	year={1967},
	publisher={Springer}
}

@article{mallik2016conforming,
	title={Conforming finite element methods for the von K{\'a}rm{\'a}n equations},
	author={Mallik, Gouranga and Nataraj, Neela},
	journal={Advances in Computational Mathematics},
	volume={42},
	number={5},
	pages={1031--1054},
	year={2016},
	publisher={Springer}
}

@article{mallik2016nonconforming,
	title={A nonconforming finite element approximation for the von {K{\'a}rm{\'a}n} equations},
	author={Mallik, Gouranga and Nataraj, Neela},
	journal={ESAIM: Mathematical Modelling and Numerical Analysis},
	volume={50},
	number={2},
	pages={433--454},
	year={2016},
	publisher={EDP Sciences}
}

@article{miyoshi1976mixed,
	title={A mixed finite element method for the solution of the von {K{\'a}rm{\'a}n} equations},
	author={Miyoshi, Tetsuhiko},
	journal={Numerische Mathematik},
	volume={26},
	number={3},
	pages={255--269},
	year={1976},
	publisher={Springer}
}

@article{muradova2007unilateral,
	title={A unilateral contact model with buckling in von {K{\'a}rm{\'a}n}
	plates},
	author={Muradova, Aliki D and Stavroulakis, Georgios E},
	journal={Nonlinear Analysis: Real World Applications},
	volume={8},
	number={4},
	pages={1261--1271},
	year={2007},
	publisher={Elsevier}
}

@article{ohtake1980analysisI,
	title={Analysis of certain unilateral problems in von {K{\'a}rm{\'a}n} plate
	theory by a penalty method-part 1. A variational principle with
	penalty},
	author={Ohtake, K and Oden, J Tinsley and Kikuchi, Noboru},
	journal={Computer Methods in Applied Mechanics and Engineering},
	volume={24},
	number={2},
	pages={187--213},
	year={1980},
	publisher={Elsevier}
}

@article{ohtake1980analysisII,
  title={Analysis of certain unilateral problems in von K{\'a}rm{\'a}n plate theory by a penalty method-part 2. approximation and numerical analysis},
  author={Ohtake, Kunihiko and Oden, J Tinsley and Kikuchi, Noboru},
  journal={Computer Methods in Applied Mechanics and Engineering},
  volume={24},
  number={3},
  pages={317--337},
  year={1980},
  publisher={Elsevier}
}

@article{quarteroni1979hybrid,
	title={Hybrid finite element methods for the von {K{\'a}rm{\'a}n} equations},
	author={Quarteroni, Alfio},
	journal={Calcolo},
	volume={16},
	number={3},
	pages={271--288},
	year={1979},
	publisher={Springer}
}

@article{reinhart1982numerical,
	title={On the numerical analysis of the von {K{\'a}rm{\'a}n} equations: mixed finite element approximation and continuation techniques},
	author={Reinhart, Laure},
	journal={Numerische Mathematik},
	volume={39},
	number={3},
	pages={371--404},
	year={1982},
	publisher={Springer}
}

@book{suttmeier2008numerical,
	title={Numerical solution of variational inequalities by adaptive finite elements},
	author={Suttmeier, Franz-Theo},
	year={2008},
	publisher={Springer}
}

@article{yau1992obstacle,
	title={Obstacle problem for von {K{\'a}rm{\'a}n} equations},
	author={Yau, Shing-Tung and Gao, Yang},
	journal={Advances in Applied Mathematics},
	volume={13},
	number={2},
	pages={123--141},
	year={1992},
	publisher={Elsevier}
}

\end{document}